\documentclass[11pt,a4paper, oneside]{amsart}

\usepackage{geometry}
\usepackage{amsmath}
\usepackage{amsthm}
\usepackage{amssymb}
\usepackage{amsfonts}
\usepackage{enumerate}

\usepackage{graphicx}

\theoremstyle{plain}
\newtheorem{theorem}{Theorem}[section]
\newtheorem{lemma}[theorem]{Lemma}
\newtheorem{proposition}[theorem]{Proposition}
\newtheorem{corollary}[theorem]{Corollary}

\theoremstyle{definition}
\newtheorem{definition}[theorem]{Definition}

\theoremstyle{remark}
\newtheorem{example}[theorem]{Example}
\newtheorem{remark}[theorem]{Remark}

\newcommand{\floor}[1]{\lfloor #1 \rfloor}

\begin{document}

\title{Defining rough sets using tolerances compatible with an equivalence}

\author[J.~J{\"a}rvinen]{Jouni J{\"a}rvinen}

\address[J.~J{\"a}rvinen]{Department of Mathematics and Statistics, University of Turku, 20014~Turku, Finland}
\email{jjarvine@utu.fi}

\author[L.~Kov\'{a}cs]{L\'{a}szl\'{o} Kov\'{a}cs}
\address[L.~Kov\'{a}cs]{Institute of Information Science, University of Miskolc, 3515~Miskolc-Egyetemv\'{a}ros, Hungary} 
\email{kovacs@iit.uni-miskolc.hu}

\author[S.~Radeleczki]{S\'andor Radeleczki}
\address[S.~Radeleczki]{Institute of Mathematics, University of Miskolc, 3515~Miskolc-Egyetemv\'{a}ros, Hungary}
\email{matradi@uni-miskolc.hu}


\begin{abstract}
We consider tolerances $T$ compatible with an equivalence $E$ on $U$, meaning that the relational product 
$E \circ T$ is included in $T$. We present the essential properties of $E$-compatible tolerances and
study rough approximations defined by such $E$ and $T$. We consider rough set pairs $(X_E,X^T)$,
where the lower approximation $X_E$ is defined as is customary in rough set theory, but $X^T$ allows more elements to be possibly in $X$
than $X^E$. Motivating examples of $E$-compatible tolerances are given, and the essential lattice-theoretical properties of the ordered set of rough sets 
$\{ (X_E,X^T) \mid X \subseteq U\}$ are established.
\end{abstract}


\maketitle

\section{Introduction}

Rough sets were introduced by Z. Pawlak in \cite{Pawl82}. He was assuming that our knowledge about the objects 
of a universe $U$ is given in the terms of an equivalence $E$ on $U$. 
In rough set theory, equivalences are treated as indistinguishability relations. Indistinguishability of objects $x$ 
and $y$ means that we do not have a way to distinguish $x$ and $y$ based on our information. Indistinguishability relations 
are hence assumed to be reflexive, symmetric, and transitive.

A \emph{tolerance relation} (or simply \emph{tolerance}) is a reflexive and symmetric binary relation.
In this work we treat tolerances as similarity relations. This means that we do not assume similarity to be transitive. 
For instance in \cite{Luce56} is given this example justifying non-transitivity:
``Find a subject who prefers a cup of coffee with one cube of sugar to one with five cubes 
(this should not be difficult). Now prepare 401 cups of coffee with $(1 + i / 100) x$ grams of sugar, $i = 0, 1,\ldots, 400$, 
where $x$ is the weight of one cube of sugar. It is evident that he will be indifferent between cup $i$ and cup $i + 1$, for any $i$, 
but by choice he is not indifferent between $i = 0$ and $i = 400$.'' In fact, there are also opinions that similarity relation 
should be only reflexive, not symmetric, because similarity can be sometimes seen directional. As noted in \cite{Tversky77}:  
`We say ``the portrait resembles the person'' rather than ``the person resembles the portrait''.'. 
However, in this work we assume that similarity relations are tolerances.

In the last decades, several extensions of the basic rough set model were proposed in the research literature.  
The main motivation of these extensions was to provide efficient modelling of imprecise or missing data values. 
There are early articles from 1980s and 1990s, in which rough approximations are defined in terms of tolerances.
For instance,  E.~Or{\l}owska and Pawlak considered in \cite{OrlPaw84} so-called ``nondeterministic information systems'' 
in which attribute values of objects may be sets instead of single values.
By using such information it is possible to define tolerance relations representing similarity of objects.
In addition, J.A.~Pomyka{\l}a \cite{Pomy94} and B.~Konikowska \cite{Konikowska1997} have considered approximation 
operations defined by strong similarity relations of nondeterministic information systems. 
Also in \cite{Skowron1996} equivalences were replaced by tolerances to represent our knowledge about the objects.

First systematic studies on different types of binary relation (including tolerances) was given in 
\cite{Yao96}. Or{\l}owska has studied so-called information relations reflecting distinguishability or
indistinguishability of the elements of the universe of discourse in \cite{Orlowska94,Orlowska98}.
They were also considered in \cite{Zhao2007}. It should also be noted that tolerances are
closely related to set-coverings, and rough approximation defined by coverings are studied for the first time by
W.~\.{Z}akowski in \cite{Zak83}. In \cite{Yao2012}, a review of covering based rough set approximations is presented.
Furthermore, authors of this paper have considered lattice-theoretical properties of rough sets defined by
tolerances, for example, in \cite{Jarv01,JarRad14,JarRad18}. 

The limitations of the single-equivalence approach were analysed among others in \cite{Qian2010}, where so-called multi-granulation 
rough set model was introduced. In that paper, for two equivalences $P$ and $Q$ on $U$, the lower and upper approximations of 
$X \subseteq U$ were defined, respectively, as
\begin{align*}
  X_{P + Q } &=  \{  x \in U \mid P(x) \subseteq X \text{ or } Q(x) \subseteq X \} \\
\intertext{and} 
  X^{P + Q} &=  \{  x \in U \mid P(x) \cap X \neq \emptyset \text{ and } Q(x) \cap X \neq \emptyset \}.
\end{align*}
Here $R(x) = \{ y \in U \mid x \, R \, y\}$ for any binary relation $R$ on $U$ and $x \in U$.
An extension of this approach was given in \cite{Xu2012}, where a finite family of disjoint subsets of the attribute set of an 
information system is used to define the approximations. This paper also investigates some measures, such as the quality and 
the precision of approximation. Multi-granulation of fuzzy rough sets was presented in \cite{Xu2014}.
Relationships between relation-based rough sets and covering-based rough sets are investigated in \cite{ResGom14}.

Our approach differs from the ones appearing in the literature, because our main idea is that the lower approximation of sets
are defined in terms of the equivalence $E$ and the upper approximations are defined in terms of a tolerance $T$ compatible with $E$, 
that is, $E \circ T \subseteq T$. This condition means that if $x$ and $y$ are $E$-indistinguishable and $y$ is $T$-similar with some $z$, then 
this $z$ is $T$-similar to $x$ also. 

We start with the following definitions. For $X \subseteq U$, the \emph{lower approximation} of $X$ is defined as
\[
X_R = \{x \in U \mid R(x) \subseteq X\},
\]
and the \emph{upper approximation} of $X$ is given by
\[
X^R = \{x \in U \mid R(x) \cap X \neq \emptyset \}.
\]
Let us now recall from literature \cite{Jarv07} some essential properties of these approximations. We denote
by $\wp(U)$ the \emph{power set} of $U$, that is, $\wp(U) = \{ X \mid X \subseteq U\}$. Let $\mathcal{H} \subseteq \wp(U)$
be a family of subsets of $U$. Then,
\[ \Big ( \bigcup_{X \in \mathcal{H}} X \Big )^R = \bigcup_{X \in \mathcal{H}} X^R \quad \text{ and } \quad 
   \Big ( \bigcap_{X \in \mathcal{H}} X \Big )_R = \bigcap_{X \in \mathcal{H}} X_R . \]
If $R$ is reflexive, then $X_R \subseteq X \subseteq X^R$ and we can partition the elements of $U$ into three
disjoint classes with respect to the set $X$:

\begin{enumerate}[\rm (1)]
\item The elements which are certainly in $X$. These are interpreted as
the elements in $X_R$, because if $x \in X_R$, then all the
elements to which $x$ is $R$-related are in $X$.

\item The elements which certainly are not in $X$. These are elements $x$
of $U$ such that all the elements to which $x$ is $R$-related are
not in $X$, that is, $R(x) \cap X = \emptyset$, or equivalently,
$R(x) \subseteq X^c$, where $X^c$ is the \emph{complement} of $X$, that is,
$X^c = U \setminus X$.

\item The elements whose belonging in $X$ cannot be decided in terms of the
knowledge $R$. These are the elements $x \in U$ which are $R$-related at least with one element 
of $X$ and also with at least one element from $X$'s complement $X^{c}$. 
In other words, $R(x) \cap X \neq \emptyset$ and $R(x) \nsubseteq X$, that is, 
$x\in X^R \setminus X_R$.
\end{enumerate}

Let $T$ be a tolerance on $U$. It is known \cite{Jarv99} that the pair $({_T},{^T})$ is an order-preserving
Galois connection on $\wp(U)$. From this fact it follows that for any $X \subseteq U$,
\[ 
(X_T)^T \subseteq X \subseteq (X^T)_T, \quad ((X_T)^T)_T = X_T, \quad ((X^T)_T)^T = X^T.
\]
Moreover, if we denote
\[ \wp(U)^T = \{X^T \mid X \subseteq U\} \quad \text{and} \quad  \wp(U)_T = \{X_T \mid X \subseteq U\}, \]
then the ordered set $(\wp(U)^T,\subseteq)$ is a complete lattice such that for any $\mathcal{H} \subseteq \wp(U)$,
\begin{equation}\label{Eq:UpLattice}
\bigvee_{X \in \mathcal{H}} X^T = \bigcup_{X \in \mathcal{H}} X^T 
\quad \mbox{ and  } \quad 
\bigwedge_{X \in \mathcal{H}} X^T 
=  \Big ( \Big ( \bigcap_{X \in \mathcal{H}} X^T \Big )_T \Big )^T .
\end{equation}
Analogously, $(\wp(U)_T,\subseteq)$ is a complete lattice such that for any $\mathcal{H} \subseteq \wp(U)$,
\begin{equation}\label{Eq:DownLattice}
\bigvee_{X \in \mathcal{H}} X_T = \Big ( \Big ( \bigcup_{X \in \mathcal{H}} X_T \Big )^T \Big )_T
\quad \mbox{ and  } \quad 
\bigwedge_{X \in \mathcal{H}} X_T = \bigcap_{X \in \mathcal{H}} X_T  .
\end{equation}
Let us also note that if $E$ is an equivalence, then $\wp(U)_E = \wp(U)^E$. We present more properties of this
complete lattice in Section~\ref{Sec:CombatileTolerances} while considering $E$-definable sets.

In this work, we define the rough set of set $X \subseteq U$ as a pair $(X_E, X^T)$. The idea behind
studying pairs $(X_E,X^T)$, where $T$ is an $E$-compatible tolerance, is that the equivalence $E$ represents 
``strict'' information and the information represented by $T$ is ``soft''. Hence $X_E$ is defined as it is usual in 
rough set theory, but $X^T$ is now more permissible, because $E \subseteq T$ and thus $X^E \subseteq X^T$. Additionally, 
we have $(X^T)^E = X^T$, meaning that $X^T$ is a union of $E$-classes. 

The set of $(E,T)$-rough sets is $\mathit{RS}(E,T) = \{ (X_E, X^T) \mid X \subseteq U\}$ and
$\mathit{RS}(E,T)$ can be ordered by the coordinatewise inclusion. We show that $\mathit{RS}(E,T)$
forms a complete lattice which is not generally distributive. Finally,
we give some conditions under which $\mathit{RS}(E,T)$ is distributive and defines a regular double $p$-algebra.

This work is structured as follows. In Section~\ref{Sec:CombatileTolerances} we give the basic properties of
$E$-compatible tolerances and rough approximations defined by them. The section ends by three subsections giving
examples from where $E$-compatible tolerances can be found. Section~\ref{Sec:Lattices} is devoted to the study of
the lattice-theoretical properties of $\mathit{RS}(E,T)$. In Section~\ref{Sec:FurtherProperties} we consider
some further properties of $\mathit{RS}(E,T)$, such as it being a completely distributive regular double 
pseudocomplemented lattice. We also study the case in which the $E$-compatible tolerance is an equivalence.
Some concluding remarks end the work.

\section{Tolerances compatible with equivalences} \label{Sec:CombatileTolerances}

If $E$ is an equivalence on $U$, we denote for any $x$ the ``$E$-neighbourhood'' $E(x)$ of $x$ by $[x]_E$, 
because this notation is conventional in the literature. The set $[x]_E$ is the \emph{equivalence class} of $x$
with respect to the equivalence relation $E$. This is also said to be the $E$-equivalence class of $x$, and often
even the $E$-class of $x$.
The \emph{quotient set} $U/E$ is the set of all equivalence classes, that is, $U/E = \{[x]_E \mid x \in U\}$.

Let $R$ and $S$ be two binary relations on $U$. The \emph{product} $R \circ S$ of the relations $R$ and $S$ is
defined by 
\[ R \circ S = \{ (x,y) \in U^2 \mid (\exists z \in U) \, x \, R \, z  \ \text{ and } \ z \, S \, y\}. \]
The following lemma connects products of relations to rough approximation operations.

\begin{lemma} \label{Lem:Products}
If $S$ and $T$ are binary relations on $U$, then for all $X \subseteq U$, 
\begin{enumerate}[\rm (a)]
\item $X^{S \circ T} = (X^T)^S$;
\item $X_{S \circ T} = (X_T)_S$. 
\end{enumerate}
\end{lemma}

\begin{proof} (a) For all $x \in U$,
\begin{align*}
x \in X^{S \circ T} & \iff (\exists  y \in U) \, x \, (S \circ T) \, y \ \& \ y \in X \\
		    & \iff (\exists  y, z \in U) \, x \, S \, z \ \& \ z \, T \, y \ \& \ y \in X \\
		    & \iff (\exists  z \in X^T) \, x \, S \, z \\
		    & \iff x \in (X^T)^S. 
\end{align*}

(b) For all $x \in U$,
\begin{align*}
x \in X_{S \circ T} & \iff (\forall  y \in U) \, x \, (S \circ T) \, y \Rightarrow y \in X \\
		    & \iff (\forall  y,z \in U) \, x \, S \, z \ \& \ z \, T \, y \Rightarrow y \in X \\
		    & \iff (\forall  z \in U) \, x \, S \, z \Rightarrow z \in X_T \\
		    & \iff x \in (X_T)_S. \qedhere
\end{align*}
\end{proof}

\begin{definition} \label{Def:Compatible}
Let $E$ be an equivalence on $U$. A tolerance $T$ on $U$ is called \emph{$E$-compatible} if
\begin{equation} \label{Eq:Main}
E \circ T \subseteq T .
\end{equation}
\end{definition}

The idea behind this definition is that $x \, E \, z$, $z \, T \, y$ leads to $x \, T \, y$.
It is clear that if $T$ is an $E$-compatible tolerance, then $E \subseteq T$. The order of $E$ and $T$ in the
relation product has no importance either. Indeed, if $E$ is an equivalence and $T$ is a tolerance on $U$, then 
$E^{-1} = E$, $T^{-1} = T$ and $(E \circ T)^{-1} = T^{-1} \circ E^{-1}$, and we have that
\begin{equation} \label{Eq:Inverse}
E \circ T \subseteq T \iff (E \circ T)^{-1} \subseteq T^{-1} \iff T^{-1} \circ E^{-1} \subseteq  T^{-1} \iff T \circ E \subseteq T.
\end{equation}
Hence, $E \circ T \subseteq T$ and $T \circ E \subseteq T$ are equivalent conditions. Because $T \subseteq E \circ T$ and 
$T \subseteq T \circ E$, we can immediately write the following characterization.

\begin{lemma} \label{Lem:Commutation}
If $E$ is an equivalence and $T$ a tolerance on $U$, then
\begin{center}
$T$ is $E$ compatible \ $\iff$ \ $E \circ T = T$ \ $\iff$ \ $T \circ E = T$. 
\end{center}
\end{lemma}

Interestingly, in the literature can be find analogous notions%
\footnote{We would like to thank an anonymous referee for pointing this out.}.
In particular, in \cite{Slowinski95} the authors consider ``similarity relations extending equivalences''.
They say that a binary relation $R$ on $U$ is a \emph{similarity relation extending an equivalence $E$} on $U$ if:
\begin{enumerate}[({Ex}1)]
 \item For all $x \in U$, $[x]_E \subseteq R(x)$.
 \item For all $x,y \in U$, $y \in R(x)$ implies $[y]_E \subseteq R(x)$.
\end{enumerate}
Note that by (Ex1), the similarity relation is reflexive, but symmetry does not follow from this definition. We can now prove that 
if a similarity relation is a tolerance, the two notions coincide.

\begin{proposition}
Let $E$ be an equivalence and $T$ a tolerance on $U$. The following are equivalent:
\begin{enumerate}[\normalfont (i)]
\item $T$ is $E$-compatible.
\item $T$ is a similarity relation extending $E$.
\end{enumerate}
\end{proposition}

\begin{proof}
(i)$\Rightarrow$(ii): Condition (Ex1) is clear since $E \subseteq T$. Suppose $y \in T(x)$ and $z \in [y]_E$. Then
$x \, T \, y$ and $y \, E \, z$, in other words $(x,z) \in T \circ E$. Because $T$ is $E$-compatible, we have $(x,z) \in T$,
that is, $z \in T(x)$. Thus, $[y]_E \subseteq T(x)$ and (Ex2) holds.

\medskip
(ii)$\Rightarrow$(i): Suppose $(x,z) \in E \circ T$. Then, there exists $y$ such that $x \, E \, y$ and $y \, T \, z$.
By (Ex2), $y \in T(z)$ implies $[y]_E \subseteq T(z)$. Because $x \in [y]_E$, we have $x \in T(z)$. Thus, $(x,z) \in T$ and
$T$ is $E$-compatible. 
\end{proof}
\noindent%
Note also that there are studies on compatibility of fuzzy relations; see \cite{Kheniche16} and the references therein.

\medskip%
Let $T$ be a tolerance on $U$. The \emph{kernel} of $T$ is defined by
\[ \ker T = \{ (x,y) \mid T(x) = T(y) \}. \]
The relation $\ker T$ is clearly an equivalence on $U$, and $\ker T \subseteq T$, because for $(x,y) \in \ker T$,
$x \in T(x) = T(y)$, that is, $(x,y) \in T$.
Our next proposition characterizes $E$-compatible tolerances.

\begin{proposition} \label{Prop:CompTolerance}
Let $E$ be an equivalence on $U$. A tolerance $T$ on $U$ is $E$-compatible if and only if 
$E \subseteq \ker T$. 
\end{proposition}

\begin{proof}
Suppose that $T$ is $E$-compatible. We show that $E \subseteq \ker T$.
Assume $(x,y) \in E$. Let $z \in T(x)$. 
Then $z \, T \, x$ and $x \, E \, y$, that is, $z \, (T \circ  E) \, y$. 
By Lemma~\ref{Lem:Commutation}, $T \circ E = T$. Hence, $z \, T \, y$ and $z \in T(y)$.
We have proved that $T(x) \subseteq T(y)$. Similarly, we can show that $T(y) \subseteq T(x)$.
Therefore, $T(x) = T(y)$ and $(x,y) \in \ker T$.

On the other hand, suppose $E \subseteq \ker T$. Let $(x,y) \in E \circ T$.
Then, there is $z$ such that $x \, E \, z$ and $z \, T \, y$. Because $(x,z) \in \ker T$,
$y \in T(z) = T(x)$. Thus, $x \, T \, y$ and $T$ is $E$-compatible.
\end{proof}

Proposition~\ref{Prop:CompTolerance} means that if $x$ and $y$ are $E$-indistinguishable, also their 
$T$-neighbourhoods are the same, that is, $T(x) = T(y)$. Another consequence of 
Proposition~\ref{Prop:CompTolerance} is that $\ker T$ is the greatest equivalence with whom
the tolerance $T$ is compatible. If $F$ is an equivalence on $U$, then $\ker F = F$. 
This means that $F$ is $E$-compatible if and only if $E \subseteq F$.

We will next consider rough approximations. It is known (see e.g.\@ \cite{Jarv99}) that if $T$ 
is a tolerance on $U$, then for all $X \subseteq U$,
\[ X^T = \bigcup \{ T(x) \mid T(x) \cap X \neq \emptyset\} \]
In addition, if $E$ is an equivalence on $U$, then for any $X \subseteq U$,
\[ X_E = \bigcup \{ [x]_E \mid [x]_E \subseteq X\} .\]

By Lemmas \ref{Lem:Products} and \ref{Lem:Commutation}, we can write the following equations.

\begin{lemma} \label{Lem:AllTheSame}
Let $E$ be an equivalence on $U$ and let $T$ be an $E$-compatible tolerance. 
For all $X \subseteq U$, the following equalities hold:
\begin{enumerate}[\rm (a)]
 \item $(X^T)^E = X^{E \circ T} = X^T = X^{T \circ E} = (X^E)^T$;
 \item $(X_T)_E = X_{E \circ T} = X_T = X_{T \circ E} = (X_E)_T$. 
\end{enumerate}
\end{lemma}

Let $E$ be an equivalence on $U$. A set $X \subseteq U$ is called \emph{$E$-definable} if $X_E = X^E$. 
This means that the set of elements which certainly are in $X$ coincides with the set of elements which 
possibly are in $X$. We denote by $\mathrm{Def}(E)$ the family of $E$-sets. It is a well-known fact (see e.g. \cite{Jarv99}) 
that the following conditions are equivalent for any $X \subseteq U$:
\begin{enumerate}[(i)]
\item $X \in \mathrm{Def}(E)$;
\item $X = X^E$;
\item $X = X_E$;
\item $X = \bigcup \mathcal{H}$ for some $\mathcal{H} \subseteq U/E$;
\item $x \in X$ and $x \, E \, y$ implies $y \in X$.
\end{enumerate}
Notice that these conditions mean that $\mathrm{Def}(E) = \wp(U)_E = \wp(U)^E$. 
It is also known  (see e.g. \cite{Jarv99}) that $(\mathrm{Def}(E),\subseteq)$ is a complete lattice in which
\[ \bigwedge \mathcal{H} = \bigcap \mathcal{H} \quad \text{ and } \quad \bigvee \mathcal{H} = \bigcup \mathcal{H} \]
for all $\mathcal{H} \subseteq \mathrm{Def}(E)$. The family of sets $\mathrm{Def}(E)$ is also closed under
complementation, that is, $X^c \in \mathrm{Def}(E)$ for all $X \in \mathrm{Def}(E)$.

Let $E$ be an equivalence on $U$ and let $T$ be an $E$-compatible tolerance. By Lemma~\ref{Lem:AllTheSame}, 
we have $(X^T)^E = X^T$ and $(X_T)_E = X_T$ for any $X \subseteq U$. This means that each $X^T$ and $X_T$ is $E$-definable. 
This implies also that $(X^T)_E = X^T$ and $(X_T)^E = X_T$. Because $E$-definable sets are unions of $E$-classes, our next 
lemma gives a description of $X_T$ and $X^T$ in terms of equivalence classes of $E$.

\begin{lemma} \label{Lem:UnionOfClasses}
Let $E$ be an equivalence on $U$ and let $T$ be an $E$-compatible tolerance. For all $X \subseteq U$,
\begin{enumerate}[\rm (a)]
 \item $X^T = \bigcup \{ [x]_E \mid T(x) \cap X \neq \emptyset\}$;
 \item $X_T = \bigcup \{ [x]_E \mid T(x) \subseteq X \}$.
\end{enumerate}

\end{lemma}

\begin{proof}
(a) If $y \in X^T$, then $T(y) \cap X \neq \emptyset$ and $y \in \bigcup \{ [x]_E \mid T(x) \cap X \neq \emptyset\}$.
On the other hand, if  $y \in \bigcup \{ [x]_E \mid T(x) \cap X \neq \emptyset\}$, then there is 
$z \in X^T$ such that $y \in [z]_E$. This means that $z \in [y]_E \cap X^T$ and thus $y \in (X^T)^E = X^T$.

(b) Suppose that $y \in X_T$. Then $T(y) \subseteq X$ gives $y \in \bigcup \{ [x]_E \mid T(x) \subseteq X \}$.
Conversely, if $y \in \bigcup \{ [x]_E \mid T(x) \subseteq X \}$, then there is $z \in U$ such that $y \, E \, z$ and
$T(z) \subseteq X$. Then, $y \, E \, z$ and $z \in X_T$ give $y \in (X_T)^E = X_T$.
\end{proof}

\begin{example}
Let $E$ be an equivalence on $U$ and let $T$ be an $E$-compatible tolerance. By Lemma~\ref{Lem:UnionOfClasses}, 
\begin{align*}
T(x) = \{x\}^T = \bigcup \{ [y]_E \mid T(y) \cap \{x\} \neq \emptyset \} 
     = \bigcup \{ [y]_E \mid x \in T(y) \} 
     = \bigcup \{ [y]_E \mid y \in T(x) \}.
\end{align*}
This means that $T(x)$ is a union of $E$-classes for any $x \in U$.

In fact, an $E$-class behaves like one ``point'' with respect to the tolerance $T$. 
The situation can be depicted as in Figure~\ref{Fig:Compatible}, where $E$-classes are represented by circles. 
A line connecting two $E$-classes mean that all elements between these two classes are mutually $T$-related. 
For instance, $x_3$ is $T$-related with $x_1$ and $x_5$, but $x_1$ and $x_5$ are not  $T$-related. The
$T$-neighbourhood of $x_3$ is a union of $E$-classes, that is,
\[ T(x_3) = \{x_1,x_2\} \cup \{x_3\} \cup \{x_4,x_5,x_6\}.\]

\begin{figure}
\centering
\includegraphics[width=100mm]{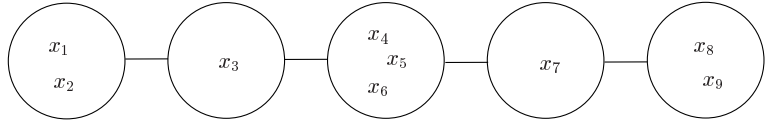}
\caption{Equivalence classes of $E$ are represented by circles. 
A line connecting two $E$-classes mean that all elements between these two classes are mutually $T$-related. 
For instance, $x_3$ is $T$-related with $x_1$ and $x_5$, but $x_1$ and $x_5$ are not  $T$-related.}
\label{Fig:Compatible}
\end{figure}
\end{example}

We end this section by three short subsections containing some motivating examples of $E$-compatible tolerances.

\subsection{Strong and weak indistinguishability relations}

An \emph{information system} in the sense of Pawlak \cite{Pawl81} is a triple $(U,A,\{V_a\}_{a \in A})$, where
$U$ is a set of \emph{objects}, $A$ is a set of \emph{attributes}, and $V_a$ is the \emph{value set} of $a \in A$. Each attribute
is a mapping $a \colon U \to V_a$ and $a(x)$ is the value of the attribute $a$ of for $x$.

For any $B \subseteq A$, the \emph{strong indistinguishability relation} of $B$ is defined by
\[ \mathrm{ind}(B) = \{ (x,y) \mid (\forall a \in B) \, a(x) = a(y) \}. \]
In the literature, strong indistinguishability relations are commonly called ``indiscernibility relations''. The 
\emph{weak indistinguishability relation} of $B$ is given by 
\[ \mathrm{wind}(B) = \{ (x,y) \mid (\exists a \in B) \, a(x) = a(y) \}. \]
Let us denote for any $a \in A$, the relation $\mathrm{ind}(\{a\})$ simply by $\mathrm{ind}(a)$. It is obvious that for all $B \subseteq A$,
\[  \mathrm{ind}(B) = \bigcap_{a \in B} \mathrm{ind}(a) \quad \text{ and } \quad  \mathrm{wind}(B) = \bigcup_{a \in B} \mathrm{ind}(a). \]
It is also clear the for any $\emptyset \neq B \subseteq A$, $\mathrm{ind}(B)$ is an equivalence and $\mathrm{wind}(B)$ is a tolerance
on $U$. Additionally, we can write the following lemma.

\begin{lemma} \label{Lem:InfSyst}
Let $(U,A,\{V_a\}_{a \in A})$ be an information system and $\emptyset \neq B \subseteq A$. Then  $\mathrm{wind}(B)$ is  $\mathrm{ind}(B)$-compatible.
\end{lemma}

\begin{proof}
Suppose $(x,y) \in {\rm ind}(B) \circ \mathrm{wind}(B)$. Then there is $z \in U$ such that $a(x)  = a(z)$ for all $a \in B$ and there is
$b \in B$ such that $b(z) = b(y)$. Since $b \in B$, $b(x) = b(z)$. This implies $b(x) = b(z) = b(y)$ and $(x,y) \in \mathrm{wind}(B)$.
\end{proof}

\begin{example}\label{Ex:SandorMedical}
Suppose that $U$ is a set of people and the set $A$ of attributes consists of results of the medical test that can be performed
in a hospital for patients. For instance, $a \in A$ can be the attribute ``blood pressure'' and $a(x) = \text{``normal''}$ means that
the patient $x$ has blood pressure readings in the range from 120 over 80 (120/80) to 140 over 90 (140/90).

Let $X \subseteq U$ be a set of people which are known to have some illness. Let $B \subseteq A$ be a set of medical tests whose 
results are relevant in the diagnostics of the disease $X$. The lower approximation $X_{\mathrm{ind}(B)}$ consists of patients that 
certainly have the illness $X$. If $x \in X_{\mathrm{ind}(B)}$ and $(x,y) \in \mathrm{ind}(B)$, then also $y$ has the illness $X$, 
because all people having the same symptoms as $y$ are known to be sick. On the other hand, $X^{\mathrm{wind}(B)}$ contains persons which potentially 
have the illness $X$, because if $x \in X^{\mathrm{wind}(B)}$, then $x$ has at least one common meaningful symptom with a person having the illness $X$. 
Therefore, if $x \in X^{\text{wind}(B)}$, then we cannot exclude the possibility that $x$ is having the illness $X$.
\end{example}

\subsection{Tolerances induced by coverings} \label{SubSec:Coverings}

A collection $\mathcal{C}$ of nonempty subsets of $U$ is a \emph{covering} of $U$ if $\bigcup \mathcal{C} = U$. 
Each covering $\mathcal{C}$ of $U$ defines a tolerance
\[ T_\mathcal{C} = \{ (x,y) \mid (\exists B \in \mathcal{C}) \, x,y \in B\} \]
on $U$, called the \emph{tolerance induced by $\mathcal{C}$}.
The following lemma is well-known, but we give its proof for the sake of completeness.

\begin{lemma} \label{Lem:Covering}
Let $\mathcal{C}$ be a covering of $U$ and denote $T = T_\mathcal{C}$. For any $X \subseteq U$,
\[ X^T = \bigcup \{ B \in \mathcal{C} \mid B \cap X \neq \emptyset \}. \]
\end{lemma}

\begin{proof}
Assume $x \in X^T$, that is, $T(x) \cap X \neq \emptyset$. This means that there is $y \in X$ such that $x \, T \, y$. Hence there is
$B \in \mathcal{C}$ which contains both $x$ and $y$. We have that $B \cap X \neq \emptyset$ and $x \in B$. So, 
$x \in  \bigcup \{ B \in \mathcal{C} \mid B \cap X \neq \emptyset \}$.

On the other hand, suppose $x \in  \bigcup \{ B \in \mathcal{C} \mid B \cap X \neq \emptyset \}$. This means that $x \in B$ for some
$B \in \mathcal{C}$ such that $B \cap X \neq \emptyset$. Therefore, there is an element $y$ in $B \cap X$. Now $y \in T(x)$ and $y \in X$. We have
$x \in X^T$.
\end{proof}

Let $T$ be a tolerance on $U$. A nonempty subset $X$ of $U$ is a \emph{$T$-preblock} if $X \times X \subseteq T$. 
Note that if $B$ is a $T$-preblock, then $B \subseteq T(x)$ for all $x \in B$. A \emph{$T$-block} is a $T$-preblock that is 
maximal with respect to the inclusion relation. Each tolerance $T$ is completely determined by its blocks, that is, 
$a \, T \, b$ if and only if there exists a block $B$ such that $a, b \in B$. In addition, if $B$ is a block, then
\begin{equation}\label{Eq:Block}
B = \bigcap_{x \in B} T(x) . 
\end{equation}
We may characterize tolerances compatible with an equivalence in terms of tolerance blocks.

\begin{proposition}
Let $E$ be an equivalence on $U$. A tolerance $T$ on $U$ is $E$-compatible if and only if each $T$-block is $E$-definable.
\end{proposition}

\begin{proof} Suppose $T$ is $E$-compatible and let $B$ be a $T$-block. If $x \in B$ and $x \, E \, y$, then
$(x,y) \in E \subseteq \ker T$  implies $T(x) = T(y)$. Since $x \in B$, we have $B \subseteq T(x) = T(y)$. Now
$B \subseteq T(y)$ means that $B \cup \{y\}$ is a $T$-preblock containing $B$. Because $B$ is a block, we obtain
$B \cup \{y\} = B$ and $y \in B$. Thus, $B$ is $E$-definable.

On the other hand, assume that each $T$-block is $E$-definable. Suppose $x \, T \, y$ and $y \, E \, z$. Because $x \, T \, y$,
there is a $T$-block $B$ such that $x,y \in B$. Because $B$ is  $E$-definable by assumption, $y \, E \, z$ gives $z \in B$.
Since $B$ is a $T$-block, $x \, T \, z$ holds. Hence, we have shown that $T$ is $E$-compatible.
\end{proof}

A covering $\mathcal{C}$ is \emph{irredundant} if $\mathcal{C} \setminus\{B\}$ is not a covering of $U$ for any $B \in \mathcal{C}$.
Note that if $\mathcal{C}$ is an irredundant covering, then for any $B \in \mathcal{C}$ there exists an element $x$ which does not 
belong to any other set in $\mathcal{C}$, that is, $x \notin \bigcup ( \mathcal{C} \setminus \{B\})$.
Obviously, each equivalence $E$ on $U$ is such that its equivalence classes $U/E$ form an irredundant covering of $U$ and that 
the ``tolerance'' induced by $U/E$ is $E$. Tolerances induced by an irredundant covering of $U$ play an important role 
in Section~\ref{Sec:FurtherProperties}. 
It is known (see \cite{JarRad18,JarRad19}) that if $T$ is a tolerance induced by an irredundant covering, then this
covering is $\{ T(x) \mid \text{$T(x)$ is a block} \}$.

\begin{lemma} \label{Lem:CoveringKernel}
Let $T$ be a tolerance induced by an irredundant covering $\mathcal{C}$ of $U$. Then,
\[ \ker T = \{ (x,y) \mid (\forall B \in \mathcal{C}) \, x \in B \iff y \in B\} .\]
\end{lemma}

\begin{proof}
Let us denote $E_\mathcal{C} =  \{ (x,y) \mid (\forall B \in \mathcal{C}) \, x \in B \iff y \in B\}$.
If $(x,y) \in E_\mathcal{C}$, then $\{ B \in \mathcal{C} \mid x \in B\}$ is equal to  $\{ B \in \mathcal{C} \mid y \in B\}$.
This implies 
\[
T(x) = T_\mathcal{C}(x) = \bigcup \{ B \in \mathcal{C} \mid x \in B\} 
    = \bigcup \{ B \in \mathcal{C} \mid y \in B \} = T_\mathcal{C}(y) = T(y).
\]
Thus $(x,y) \in \ker T$. On the other hand, suppose that  $(x,y) \in \ker T$, which means that $T(x) = T(y)$.
Suppose that there is $B \in \mathcal{C}$ such that $x \in B$, but $y \notin B$. Because $\mathcal{C}$ is an
irredundant covering, there is $z \in B$ such that $z \notin \bigcup ( \mathcal{C} \setminus \{B\})$. This gives
that $z \in T(x) = \bigcup \{B \in \mathcal{C} \mid x \in B\}$, but $z \notin T(y) =  \bigcup \{B \in \mathcal{C} \mid y \in B\}$,
a contradiction. Therefore, for all $B \in \mathcal{C}$, $x \in B$ implies $y \in B$. Similarly, we can show
that $y \in B$ implies $x \in B$ for all $B \in \mathcal{C}$. Thus, 
\[\{ B \in \mathcal{C} \mid x \in B\} =  \{ B \in \mathcal{C} \mid y \in B\},\]
and $(x,y) \in E_\mathcal{C}$. We have now proved $\ker T =  E_\mathcal{C}$.
\end{proof}

\begin{example} \label{Ex:Venn}
If $T$ is a tolerance induced by an irredundant covering $\mathcal{C}$ of $U$, then $\ker T$ can be illustrated by a
``Venn diagram'' of $\mathcal{C}$. The equivalence classes of $\ker T$ are the ``distinct'' areas in the
diagram. For instance, if $\mathcal{C} = \{B_1,B_2,B_3\}$ is the irredundant covering depicted in
Figure~\ref{Fig:Covering} and $T$ is induced by $\mathcal{C}$, then $\ker T$ has seven equivalence classes
$c_1, c_2, \ldots, c_7$. 

Let us denote $E = \ker T$. For any $X \subseteq U$, $X_E$ is the union of the classes $c_i$ which are included in $X$ 
and $X^T$ is the union the $B_i$-sets which with intersect $X$.

\begin{figure}
\centering
\includegraphics[width=55mm]{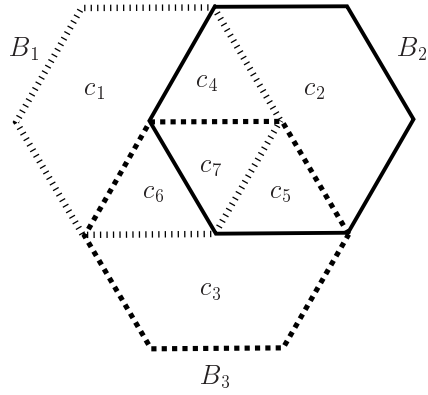}
\caption{A tolerance $T$ induced by an irredundant covering $\mathcal{C} = \{B_1,B_2,B_3\}$.
The equivalence classes of $\ker T$ are the ``distinct'' areas $c_1, \ldots, c_7$ of the diagram.}
\label{Fig:Covering}
\end{figure}
\end{example}

\subsection{Tolerances as similarity relations}

Let $(U,A,\{V_a\}_{a \in A})$ be an information system such that $V_a \subseteq \mathbb{R}$ for each $a \in A$,
where $\mathbb{R}$ denotes the set of real numbers. Suppose that for any $a \in A$, there exists a threshold
$\varepsilon_a \geq 0$ which is interpreted so that the objects $x$ and $y$ are $a$-\emph{similar} if and only 
if $a(x)$ and $a(y)$ differ from each other by at most $\varepsilon_a$.

Suppose $B \subseteq A$. We define
\[ \mathrm{sim}(B) = \{ (x,y) \mid (\forall a \in B) \, | a(x) - a(y) | \leq \varepsilon_a \}, \]
where $|x|$ denotes the absolute value of $x \in \mathbb{R}$. Note that if $\varepsilon_a = 0$ for all $a \in B$, 
then $\mathrm{sim}(B) = \mathrm{ind}(B)$.

\begin{lemma} \label{Lem:SimilarityRelation}
Let $(U,A,\{V_a\}_{a\in A})$ be an information system such that $V_a \subseteq \mathbb{R}$ for each $a \in A$.
For any $B \subseteq A$, $\mathrm{sim}(B)$ is $\mathrm{ind}(B)$-compatible.
\end{lemma}

\begin{proof}
Suppose $(x,y) \in \mathrm{ind}(B) \circ \mathrm{sim}(B)$. Thus there exists $z \in U$ such that
$(x,z) \in \mathrm{ind}(B)$ and $(z,y) \in \mathrm{sim}(B)$. This means that for all $a \in B$,
\[ a(x) = a(z)  \qquad \text{ and } \qquad  | a(z) - a(y) | \leq \varepsilon_a. \]
This implies that $|a(x) - a(y) | \leq \varepsilon_a$ for every $a \in B$. Thus, $(x,y) \in \mathrm{sim}(B)$.
\end{proof}

\begin{example}
Let  $(U,A,\{V_a\}_{a\in A})$ be an information system and let $\emptyset \neq B \subseteq A$.
Because $(x,y) \in \textrm{ind}(B)$ if and only if $a(x) = a(y)$ for all attributes $a \in B$, then actually
every tolerance $\textrm{tol}(B)$ on $U$ defined in terms of some attributes in $B$ is $\textrm{ind}(B)$-compatible.

Namely, suppose $(x,y) \in \textrm{tol}(B)$. Then $a(x)$ and $a(y)$ are ``somehow related'' with respect to some attribute(s) $a$ of $B$.
If $(y,z) \in \textrm{ind}(B)$, then $a(y) = a(z)$ means that also $a(x)$ and $a(z)$ are analogously related.

As an example, we consider  ``graded similarity''. Let $B \subseteq A$ and $k$ be an integer such that 
$0 < k \leq |B|$. Note that here $|B|$ denotes the cardinality of set $B$, and this notation should not be confused with the notation 
of absolute value of a real number. We may set:
\[
(x,y) \in  \textrm{tol}(B) \iff \text{there is $C \subseteq B$ such that $|C| = k$ and $a(x) = a(y)$ for all $a \in C$}.
\]
This means that $x$ and $y$ have same value for $k$ attributes of $B$.
\end{example}

\begin{example}
This example demonstrates that if tolerances $T_1,T_2,T_3$ are $E$-compatible tolerances such that $T_1 \subseteq T_2 \subseteq T_3$, 
then $X_E \subseteq X \subseteq X^{T_1} \subseteq X^{T_2} \subseteq X^{T_3}$ for every $X \subseteq U$.
Here $U = \mathbb{R}^2$, and for $a \in U$, $a.x$ denotes the $x$-coordinate and
$a.y$ denotes the $y$-coordinate of $a$. Let us define the following tolerances:
\begin{align*}
T_1 & = \{ (a,b) \mid \max(|\floor{a.x} - \floor{b.x}|, | \floor{a.y} - \floor{b.y}|) \leq 1 \}; \\
T_2 & = \{ (a,b) \mid \max(|\floor{a.x} - \floor{b.x}|, | \floor{a.y} - \floor{b.y}|) \leq 3 \}; \\
T_3 & = \{ (a,b) \mid \max(|\floor{a.x} - \floor{b.x}|, | \floor{a.y} - \floor{b.y}|) \leq 6 \}.
\end{align*}
Here $\floor{x}$ denotes the ``floor'' of $x$, that is, the greatest integer less than or equal to $x$. Note that 
$\floor{-1.1} = -2$, for example. Clearly, $T_1 \subseteq T_2 \subseteq T_3$.

We define an equivalence $E$ on $U$ by
\[
E =\{ (a,b) \mid \max(| \floor{a.x} - \floor{b.x}|, | \floor{a.y} - \floor{b.y}|) = 0 \}.
\]
The tolerances $T_1$, $T_2$, $T_3$ are obviously $E$-compatible. 
A set $X$ is defined as a sphere:
\[
 	X = \{ (a,b) \mid d( (a,b), (a_0,b_0) ) = r_0 \}.
\]
Here $d \colon U \times U \to [0, \infty)$ is a distance function, $ (a_0,b_0)$ is some fixed point in $U$, 
and $r_0$ is some real-constant. In Figure~\ref{Fig:MultipleBorders}, the set $X$ is denoted by a white line, 
the central dark area is the lower approximation $X_E$, and the three grey layers of different intensity show 
the upper approximations of $X$ in terms of $T_1$, $T_2$, and $T_3$. 

\begin{figure}
\centering
\includegraphics[keepaspectratio, width=85mm]{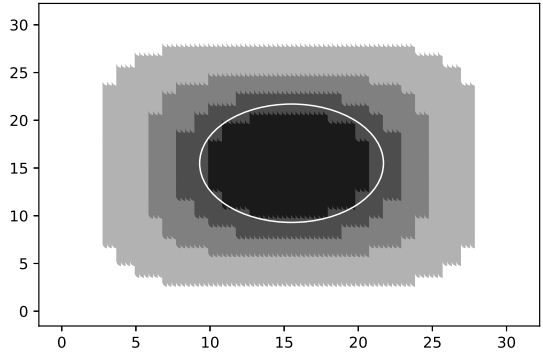}
\caption{The central dark area is the lower approximation $X_E$, and the three grey layers of different intensity show 
the upper approximations of $X$ in terms of $T_1$, $T_2$, and $T_3$.}
\label{Fig:MultipleBorders}
\end{figure}
\end{example}

\section{Lattices of rough sets based on two relations} \label{Sec:Lattices}

Let $E$ be an equivalence on $U$. We define a relation $\equiv_E$ on $\wp(U)$ by setting
\[ X \equiv_E Y \iff X_E = Y_E \quad \text{and} \quad X^E = Y^E . \]
The relation $\equiv_E$ is called \emph{rough $E$-equality} and according to Pawlak's original definition \cite{Pawl82},
the equivalence classes of $\equiv$ are called \emph{$E$-rough sets}.

Each rough set $\mathcal{R} \in \wp(U) / {\equiv_E}$ is completely defined by the pair
$(X_E,X^E)$, where $X \in \mathcal{R}$. Therefore, each $E$-rough set can be equivalently viewed as this kind of
pair, and we call the set
\[ \mathit{RS}(E) = \{ (X_E, X^E) \mid X \subseteq U \} \]
as the set of \emph{$E$-rough sets}.

A \emph{complete sublattice} of a complete lattice $L$ is a nonempty set $H \subseteq L$ such that 
$\bigvee_L S$ and $\bigwedge_L S$ belong to $H$ for every $S \subseteq H$. If $L$ and $K$ are complete lattices, 
then the \emph{Cartesian product}
\[ L \times K = \{ (a,b) \mid a \in L \text{ \ and \ } b \in K \} \]
forms a complete lattice such that
\[ \bigvee_{i \in I} (a_i,b_i ) = \Big ( \bigvee_{i \in I} a_i, \bigvee_{i \in I} b_i \Big )
 \text{ \ and \ }
   \bigwedge_{i \in I} (a_i,b_i) = \Big ( \bigwedge_{i \in I} a_i, \bigwedge_{i \in I} b_i \Big )
\]
for all $\{ (a_i,b_i) \mid i \in I \} \subseteq L \times K$. Note that the order of $L \times K$ is given by
\[ (a_1,b_1) \leq (a_2,b_2) \iff a_1 \leq_L a_2 \text{ \ and \ } b_1 \leq_K b_2 .\] 
This order is called \emph{coordinatewise order}.
 
It is known (see e.g.\@ \cite{PomPom88}) that $\mathit{RS}(E)$ is a complete sublattice of 
$\mathrm{Def}(E) \times \mathrm{Def}(E)$, that is, for any $\mathcal{H} \subseteq \wp(U)$,
\[ \bigwedge_{X \in \mathcal{H}} (X_E, X^E) 
= \Big (\bigcap_{X \in \mathcal{H}} X_E, \bigcap_{X \in \mathcal{H}}X^E \Big ) 
\qquad \text{ and } \qquad 
\bigvee_{X \in \mathcal{H}} (X_E,X^E) 
= \Big (\bigcup_{X \in \mathcal{H}} X_E, \bigcup_{X \in \mathcal{H}}X^E \Big ) . \]
Even $X_E \subseteq X^E$ for any $X \subseteq U$, not every pair $(A,B)$ such that $A,B \in \mathrm{Def}(E)$ and $A \subseteq B$
does not form a rough set. The following characterization is by P.~Pagliani \cite{Pagl97}:
\[ \mathit{RS}(E) = \{ (A,B) \in \mathrm{Def}(E)^2 \mid A \subseteq B \text{ and } \Sigma_E \subseteq A \cup B^c \},
\]
where
\[ \Sigma_E = \{ [x]_E \mid  [x]_E  = \{x\} \, \} . \]
This means that $\Sigma_E$ contains the singleton $E$-classes. 
Note that $A \cup B^c = (B \setminus A)^c$, so $(A,B) \in \mathrm{Def}(E)^2$ belongs to $\mathit{RS}(E)$
if and only if $A \subseteq B$ and $\Sigma_E \cap (B \setminus A) = \emptyset$.

On the other hand, it is known that if $T$ is a tolerance on $U$, then the set of pairs
\[ \mathit{RS}(T) = \{ (X_T, X^T) \mid X \subseteq U \} \]
ordered by coordinatewise inclusion is not in general a lattice \cite{Jarv07}. However, if $T$ is a tolerance induced by an 
irredundant covering of $U$, then $\mathit{RS}(T)$ is a complete sublattice of $\wp(U)_T \times \wp(U)^T$,
which means that for any $\mathcal{H} \subseteq \wp(U)$,
\begin{align*}
\bigwedge_{X \in \mathcal{H}} (X_T,X^T) &= \Big (\bigcap_{X \in \mathcal{H}} X_T, \Big ( \Big ( \bigcap_{X \in \mathcal{H}}X^T \Big)_T \Big )^T  \Big ) \\
\intertext{and}
\bigvee_{X \in \mathcal{H}} (X_T,X^T) &= \Big ( \Big ( \Big ( \bigcup_{X \in \mathcal{H}} X_T \Big )^T \Big )_T , \bigcup_{X \in \mathcal{H}}X^T \Big ) . 
\end{align*}

In this section, we study the structure of the pairs
\[ \mathit{RS}(E,T) = \{ (X_E, X^T) \mid X \subseteq U \}, \]
where $E$ is an equivalence and $T$ is an $E$-compatible tolerance. We start with the following theorem.

\begin{theorem} \label{Thm:CompleteLattice}
Let $E$ be an equivalence on $U$. If $T$ is an $E$-compatible tolerance, then $\mathit{RS}(E,T)$ is a 
complete lattice such that for any $\mathcal{H} \subseteq \wp(U)$,
\begin{equation} \label{Eq:join}
\bigvee_{X \in \mathcal{H}} (X_E, X^T) 
= \Big (  \bigcup_{X \in \mathcal{H}} X_E, \bigcup_{X \in \mathcal{H}} X^T \Big ) . 
\end{equation}
and
\begin{equation} \label{Eq:meet}
\bigwedge_{X\in\mathcal{H}} (X_{E},X^T) = 
\Big (  \bigcap_{X\in\mathcal{H}} X_E, 
\Big ( \Big (\bigcap_{X \in \mathcal{H}} (X^T)_T \Big) \setminus \Sigma_E(\mathcal{H}) \Big) ^T \Big ),  
\end{equation}
where 
\[ \Sigma_E(\mathcal{H}) = \Big (  \Big(  \bigcap_{X\in\mathcal{H}} ( X^T)_T \Big )  
\setminus \bigcap_{X\in\mathcal{H}} X_E   \Big ) \cap \Sigma_E.
\]
\end{theorem}

\begin{proof}
First, we show that the right hand side of \eqref{Eq:join} belongs to $\mathit{RS}(E,T)$. 
As we noted,  
\[ \Big (   \bigcup_{X \in \mathcal{H}} X_E, \bigcup_{X \in \mathcal{H}} X^E \Big ) \]
belongs to $\mathit{RS}(E)$. This means that  there exists a set $Y \subseteq U$ with
\[ Y_{E} = \bigcup_{X\in\mathcal{H}} X_{E} \quad \text{ and } \quad Y^E = \bigcup_{X\in\mathcal{H}} X^E.\] 
Then, in view of Lemma~\ref{Lem:AllTheSame}, we have
\[ Y^T = (Y^E)^T = \Big(\bigcup_{X \in \mathcal{H}} X^E \Big )^T
= \bigcup_{X \in \mathcal{H}} ( X^E )^T = \bigcup_{X \in \mathcal{H}} X^T. \]
We have proved that
\[ \Big ( \bigcup_{X\in\mathcal{H}} X_E, \bigcup_{X\in\mathcal{H}} X^T \Big ) 
 = (Y_E, Y^T) \in RS(E, T). \]

It is clear that the right hand side of \eqref{Eq:join} is an upper bound of $\{ (X_E, X^T) \mid X \in \mathcal{H} \}$.  
Let $(Z_E, Z^T)$ be an upper bound of  $\{ (X_E, X^T) \mid X \in \mathcal{H} \}$. 
Then $\bigcup \{X_E \mid X \in \mathcal{H} \}  \subseteq Z_E$ and  $\bigcup \{ X^T \mid X \in \mathcal{H} \} 
\subseteq Z^T$ imply 
\[ \Big (  \bigcup_{X \in \mathcal{H}} X_E, \bigcup_{X \in \mathcal{H}} X^T \Big )
 \leq (Z_E, Z^T). \]
Therefore, \eqref{Eq:join} holds.

In order to show that the right side of \eqref{Eq:meet} belongs to $RS(E,T)$,
first we prove that
\[ \Big ( \bigcap_{ X \in \mathcal{H}} (X^T )_T \Big )  \setminus \Sigma_E(\mathcal{H}) \]
is $E$-definable. 

Suppose that $x \in  \bigcap \{ (X^T )_T \mid X \in \mathcal{H} \}   \setminus \Sigma_E(\mathcal{H})$ 
and $x \, E \, y$. Then $T(x) \subseteq X^T$ for all $X \in \mathcal{H}$. It is also clear that 
$y \notin \Sigma_E(\mathcal{H})$, because $y \in \Sigma_E(\mathcal{H})$ would mean $y \in \Sigma_E$, that is, 
$[y]_E = \{y\}$. Because $x \, E \, y$, we obtain $x = y$ and $x \in \Sigma_E(\mathcal{H}$), which is not 
possible by the original assumption.
We have now two possibilities: (i)~If $x \in \bigcap \{ X_E \mid X\in\mathcal{H} \}$, then 
$y \in [x]_E \subseteq X \subseteq (X^T)_T$ for all $X \in \mathcal{H}$. Therefore,
\[ y \in \Big ( \bigcap_{X\in\mathcal{H}} (X^T)_T \Big ) \setminus \Sigma_E(\mathcal{H}).\]
(ii)~If 
\[ x \in \Big ( \bigcap_{ X \in \mathcal{H}} (X^T )_T \Big ) \setminus \bigcap_{X \in \mathcal{H}} X_E , \]
then for any $z \in T(y)$, $z \, T \, y$ and $y \, E \, x$ imply $z \, T \, x$. We have $z \in T(x) \subseteq X^T$
for all $X \in \mathcal{H}$. This means $T(y) \subseteq X^T$ and $y \in (X^T)_T$ for every $X \in \mathcal{H}$. Thus,
\[ y \in \Big ( \bigcap_{X \in \mathcal{H}} (X^T)_T \Big ) \setminus \Sigma_E (\mathcal{H}). \]
Therefore, $\bigcap \{(X^T)_T \mid X \in \mathcal{H} \} \setminus \Sigma_E (\mathcal{H})$ is $E$-definable.

It is clear that $\bigcap \{ X_E \mid X \in \mathcal{H}\}$ is $E$-definable. Observe also that
\[ \bigcap_{X \in \mathcal{H}} X_E \subseteq \Big ( \bigcap_{X \in \mathcal{H}} (X^T)_T \Big ) \setminus 
\Sigma_E(\mathcal{H}).\]
Indeed, let $x \in \bigcap \{ X_E \mid X \in \mathcal{H} \}$. Then, $x \in X_E \subseteq X \subseteq (X^T)_T$ for
any $X \in \mathcal{H}$. Thus, $x \in \bigcap \{ (X^T)_T \mid X \in \mathcal{H} \}$. 
If $x \notin \Sigma_E$, then $x \notin \Sigma_E(\mathcal{H})$. 
If $x \in \Sigma_E$, then
\[ x \notin \Big ( \Big ( \bigcap_{X\in\mathcal{H}} (X^T)_T \Big ) \setminus 
        \bigcap_{X \in \mathcal{H}} X_E \Big ) \cap \Sigma_E = \Sigma_E(\mathcal{H}), \]
because $x \in \bigcap \{ X_E \mid X \in \mathcal{H} \}$. Therefore,
\[ x \in \Big ( \bigcap_{X\in\mathcal{H}} (X^T)_T \Big ) \setminus \Sigma_E(\mathcal{H}). \]

Next we observe that $\Sigma_E$ does not intersect with
\[ \Big ( \Big (\bigcap_{X \in \mathcal{H}} (X^T)_T \Big ) \setminus \Sigma_E (\mathcal{H}) \Big ) \setminus \bigcap_{X \in \mathcal{H}} X_E =
\Big ( \Big (\bigcap_{X \in \mathcal{H}} (X^T)_T \Big ) \setminus \bigcap_{X \in \mathcal{H}} X_E \Big ) \setminus \Sigma_E(\mathcal{H}), \]
because
\[ \Sigma_E(\mathcal{H}) = \Big (  \Big(  \bigcap_{X\in\mathcal{H}} ( X^T)_T \Big )  
\setminus \bigcap_{X\in\mathcal{H}} X_E  \Big ) \cap \Sigma_E.
\]
As we have noted, a pair $(A,B) \in \mathrm{Def}(E)^2$ belongs to $\mathit{RS}(E)$
if and only if $A \subseteq B$ and $\Sigma_E \cap (B \setminus A) = \emptyset$. Hence, we have now proved
that 
\[
\Big ( \bigcap_{X \in \mathcal{H}} X_E, 
\Big ( \bigcap_{X \in \mathcal{H}} (X^T)_T \Big) \setminus \Sigma_E(\mathcal{H})  \Big )
\]
belongs to $\mathit{RS}(E)$. This means that there is a set $Y \subseteq U$ with
\[ Y_E = \bigcap_{X \in \mathcal{H}} X_E \qquad \text{and} \qquad 
Y^E = \Big ( \bigcap_{X \in \mathcal{H}} (X^T)_T \Big) \setminus \Sigma_E(\mathcal{H}) .\] 
By Lemma~\ref{Lem:AllTheSame},
\[ Y^T = (Y^E)^T = \Big ( \Big ( \bigcap_{X \in \mathcal{H}} (X^T)_T \Big) \setminus \Sigma_E(\mathcal{H})  \Big )^T ,\] 
and $(Y_E,Y^T)$ belongs to $\mathit{RS}(E,T)$.

Finally, we prove \eqref{Eq:meet}. It is clear that 
\[ Y_E \subseteq \bigcap_{X \in \mathcal{H} } X_E \subseteq X_E \]
and
\[ Y^T \subseteq \Big ( \bigcap_{X \in \mathcal{H}} (X^T)_T \Big )^T 
= \Big ( \Big ( \bigcap_{X \in \mathcal{H}} X^T \Big )_T \Big )^T \subseteq ((X^T)_T)^T = X^T \]
for all $X \in \mathcal{H}$. Thus, $(Y_E,Y^T)$ is a lower bound of $\{ (X_E, X^T) \mid X \in \mathcal{H} \}$.

Suppose that $(Z_E,Z^T)$ is a lower bound of  $\{ (X_E, X^T) \mid X \in \mathcal{H} \}$. Then,
$Z_E \subseteq \bigcap \{ X_E \mid X \in \mathcal{H}\} = Y_E$ and $Z^T \subseteq \bigcap \{ X^T \mid X \in \mathcal{H} \}$.
We have
\[ Z \subseteq (Z^T)_T = \big ( \bigcap_{X \in \mathcal{H}} X^T \big )_T = \bigcap_{X \in \mathcal{H}} (X^T)_T.\]
We prove that $Z$ and $\Sigma_E(\mathcal{H})$ are disjoint. Assume by contradiction that there is $x \in Z \cap \Sigma_E(\mathcal{H})$.
Then $[x]_E = \{x\}$. Thus $x \in Z$ implies $x \in Z_E \subseteq \bigcap \{ X_E \mid X \in \mathcal{H} \}$. We get
\[ x \notin \Big ( \bigcap_{X \in \mathcal{H}} (X^T)_T \Big ) \setminus \bigcap_{X \in \mathcal{H}} X_E, \]
which means $x \notin \Sigma_E(\mathcal{H})$, a contradiction. Thus, $Z \cap \Sigma_E(\mathcal{H}) = \emptyset$. These
facts imply that 
\[ Z \subseteq \Big ( \bigcap_{X \in \mathcal{H}} (X^T)_T \Big ) \setminus \Sigma_E(\mathcal{H}) \]
and
\[ Z^T \subseteq \Big ( \Big ( \bigcap_{X \in \mathcal{H}} (X^T)_T \Big ) \setminus \Sigma_E(\mathcal{H}) \Big )^T = Y^T. \]
Thus, $(Y_E,Y^T)$ is the greatest lower bound of  $\{ (X_E, X^T) \mid X \in \mathcal{H} \}$. 
\end{proof}

By Theorem~\ref{Thm:CompleteLattice}, $\mathit{RS}(E,T)$ is always a complete join-sublattice of $\wp(U)_E \times \wp(U)^T$.
It is also obvious that if $\Sigma_E(\mathcal{H}) = \emptyset$ for all $\mathcal{H} \subseteq \wp(U)$, then $\mathit{RS}(E,T)$ is a complete 
meet-sublattice of $\wp(U)_E \times \wp(U)^T$. Therefore, $\mathit{RS}(E,T)$ is a complete sublattice of 
the Cartesian product $\wp(U)_E \times \wp(U)^T$ whenever $\Sigma_E(\mathcal{H}) = \emptyset$ for all $\mathcal{H} \subseteq \wp(U)$.
On the other hand, $\mathit{RS}(E,T)$ may be a complete sublattice of the Cartesian product $\wp(U)_E \times \wp(U)^T$ even
there is  $\mathcal{H} \subseteq \wp(U)$ such that $\Sigma_E(\mathcal{H}) \neq \emptyset$ (see Section~\ref{Sec:FurtherProperties}).

Let us denote 
\[ \Sigma_T = \{ T(x) \mid T(x) = \{x\} \} .\]
Because $T$ is $E$-compatible, we have $E \subseteq T$ and $\Sigma_T \subseteq \Sigma_E$.  We can write the following condition.

\begin{lemma} \label{Lem:CompleteSublattice}
Let $E$ be an equivalence on $U$ and let $T$ be an $E$-compatible
tolerance. If $\Sigma_E \subseteq \Sigma_T$, then $\mathit{RS}(E,T)$ is a complete sublattice of the Cartesian product
$\wp(U)_E \times \wp(U)^T$. 
\end{lemma}

\begin{proof}
It is enough to prove that for any $\mathcal{H} \subseteq \wp(U)$,  
\[ \Sigma_E(\mathcal{H}) = \Big (  \Big(  \bigcap_{X\in\mathcal{H}} ( X^T)_T \Big )  
\setminus \bigcap_{X\in\mathcal{H}} X_E   \Big ) \cap \Sigma_E
\]
is empty. Suppose that $a \in \Sigma_E(\mathcal{H})$ for some $\mathcal{H} \subseteq \wp(U)$. This means that 
$a \in \Sigma_E \subseteq \Sigma_T$.  Thus, $T(a) = \{a\}$ and $[a]_E = \{a\}$. Therefore, 
$a \in \bigcap \{ (X^T)_T \mid X \in \mathcal{H}\}$ yields that
$a \in X$ for all $X \in \mathcal{H}$ and $a \notin \bigcap \{ X_E \mid X \in \mathcal{H} \}$ gives that $a \notin X$ for
some $X \in \mathcal{H}$. Because these are contradicting, we have $\Sigma_E(\mathcal{H}) = \emptyset$.
\end{proof}

\begin{remark} \label{Rem:Alexandrov}
An element $x$ of a complete lattice $L$ is said to be \emph{compact} if, for every $S \subseteq L$, 
\[ x \leq \bigvee S \Longrightarrow x \leq \bigvee F \text{ for some finite subset $F$ of $S$}. \]
A complete lattice $L$ is said to be \emph{algebraic} if its each element can be represented as
a join of compact elements below it. It is well known that if $L$ is an algebraic lattice, 
then each complete sublattice of $L$ is algebraic. Similarly,
if $L$ and $K$ are algebraic lattices, then their Cartesian product $L \times K$ is algebraic.

A complete lattice $L$ is \emph{completely distributive} if for any doubly indexed
subset $\{x_{i,\,j}\}_{i \in I, \, j \in J}$ of $L$, 
\[
\bigwedge_{i \in I} \Big ( \bigvee_{j \in J} x_{i,\,j} \Big ) = 
\bigvee_{ f \colon I \to J} \Big ( \bigwedge_{i \in I} x_{i, \, f(i) } \Big ), \]
that is, any meet of joins may be converted into the join of all
possible elements obtained by taking the meet over $i \in I$ of
elements $x_{i,\,k}$\/, where $k$ depends on $i$. As in the case of algebraic lattices, any
complete sublattice of a completely distributive lattice is completely distributive. 
In addition, the Cartesian product of completely distributive lattices is completely distributive.

We have proved in \cite{JarRad14} that $\wp(U)_T$ and $\wp(U)^T$ are completely distributive and algebraic if and only if
$T$ is a tolerance induced by an irredundant covering. This means if $T$ is a tolerance induced by an irredundant
covering, then $\wp(U)_E \times \wp(U)^T$ is algebraic and completely distributive. 
Let $T$ be a tolerance induced by an irredundant covering of $U$ and let $E$ be an equivalence on $U$ such that
$E \subseteq \ker T$, that is, $T$ is $E$-compatible. We conclude that if $\mathit{RS}(E,T)$ is a complete sublattice 
of the Cartesian product $\wp(U)_E \times \wp(U)^T$, then $\mathit{RS}(E,T)$ is algebraic and completely distributive.

This has particular interest, because it is known that a complete lattice $L$ is isomorphic to an Alexandrov topology
if and only if $L$ is algebraic and completely distributive (see \cite[Remark 2.1.]{JarRad18}, for instance). 
An \emph{Alexandrov topology} is a topology in which the intersection of any family of open sets is open. In any topology  
the intersection of any finite family of open sets is open, but in Alexandrov topologies the restriction of finiteness
is omitted.
\end{remark}

\begin{example} \label{Exa:NonLattice}
Let $T$ a tolerance on $U =  \{1,2,3,4\}$ such that
\begin{center}
 $T(1) = \{1,2,3\}$, \quad $T(2) = \{1,2,4\}$, \quad $T(3) = \{1,3,4\}$, \quad $T(4) = \{2,3,4\}$. 
\end{center}
The kernel of $T$ is $U / \ker T = \{ \{1\}, \{2\}, \{3\}, \{4\} \}$. 
Let $E$ be an equivalence on $U$ such that $U/E = \{\{1,2\},\{3\},\{4\}\}$. Now $E$ is included in $T$,
but $T$ is not $E$-compatible, because $E \nsubseteq \ker T$.

The approximations are given in Table~\ref{Table:Approximations1}.
Note that in Table~\ref{Table:Approximations1}, sets in approximation pairs are denoted simply just as sequences of letters. 
For example, $\{1,2,4\}$ is denoted by $124$.

\begin{table}
{\small
\centering  
\begin{tabular}{ p{10mm} p{18mm} | p{15mm} p{20mm}}
   $X$       & $(X_E,X^T)$ & 
   $X$       & $(X_E,X^T)$ \\ \hline

$\emptyset$  & $(\emptyset,\emptyset)$  & $\{2,3\}$    & $(3,U)$   \\
$\{1\}$      & $(\emptyset,123)$        & $\{2,4\}$    & $(4,U)$   \\
$\{2\}$      & $(\emptyset,124)$        & $\{3,4\}$    & $(34,U)$  \\
$\{3\}$      & $(\emptyset,134)$        & $\{1,2,3\}$  & $(123,U)$ \\
$\{4\}$      & $(\emptyset,234)$        & $\{1,2,4\}$  & $(124,U)$ \\
$\{1,2\}$    & $(12,U)$                 & $\{1,3,4\}$  & $(34,U)$  \\
$\{1,3\}$    & $(3,134)$                & $\{2,3,4\}$  & $(34,U)$  \\
$\{1,4\}$    & $(4,U)$                  & $U$          & $(U,U)$   \\ \hline
\end{tabular} 
\caption{\label{Table:Approximations1} Approximations based on $E$ and $T$ of Example~\ref{Exa:NonLattice}}
}
\end{table}

The ordered set $\mathit{RS}(E,T)$ is given in Figure~\ref{Fig:NonLattice}. It is
not a join-semilattice, because the elements $(\emptyset,123)$ and  $(\emptyset,124)$
have minimal upper bounds $(3,U)$, $(12,U)$, and $(4,U)$, but not a smallest upper bound. Similarly, this ordered set is
not a meet-semilattice, because, for example, $(3,U)$ and $(4,U)$ have the maximal lower bounds $(\emptyset,123)$ and  
$(\emptyset,124)$, but not a biggest one. This example then shows that if $T$ is not an $E$-compatible tolerance,
$\textit{RS}(E,T)$ is not necessarily a semilattice.

\begin{figure}[h]
\centering
\includegraphics[width=80mm]{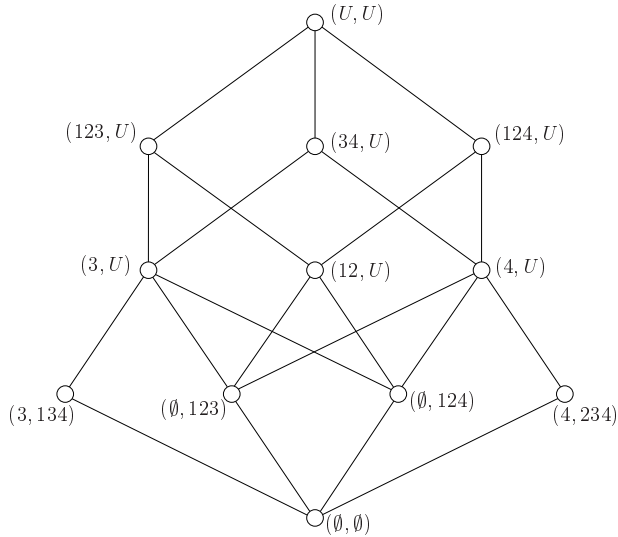}
\caption{The ordered set $\mathit{RS}(E,T)$ of Example~\ref{Exa:NonLattice} is not a semilattice. 
The elements $(\emptyset,123)$ and $(\emptyset,124)$ do not have a smallest upper bound, and $(3,U)$ and $(4,U)$ 
have no greatest lower bound.}
\label{Fig:NonLattice} 
\end{figure}
\end{example}

\begin{example} \label{Example:NonDistributiveLattice}
We denote by $\mathcal{H}$ the irredundant covering $\{ \{1,2,3\}, \{1,2,4\} \}$
of $U$. Let $T$ be the tolerance induced by $\mathcal{H}$. We have that
$T(1) = T(2) = U$, $T(3) = \{1,2,3\}$ and $T(4) = \{1,2,4\}$. The kernel of $T$ is the 
equivalence $E$ of Example~\ref{Exa:NonLattice}. Thus, the tolerance $T$ is
$E$-compatible.

\begin{table}[t]
{\small
\centering  
\begin{tabular}{ p{10mm} p{18mm} | p{15mm} p{20mm}}
   $X$       & $(X_E,X^T)$ & 
   $X$       & $(X_E,X^T)$ \\ \hline

$\emptyset$  & $(\emptyset,\emptyset)$  & $\{2,3\}$    & $(3,U)$   \\
$\{1\}$      & $(\emptyset,U)$          & $\{2,4\}$    & $(4,U)$   \\
$\{2\}$      & $(\emptyset,U)$          & $\{3,4\}$    & $(34,U)$  \\
$\{3\}$      & $(3,123)$                & $\{1,2,3\}$  & $(123,U)$ \\
$\{4\}$      & $(4,124)$                & $\{1,2,4\}$  & $(124,U)$ \\
$\{1,2\}$    & $(12,U)$                 & $\{1,3,4\}$  & $(34,U)$  \\
$\{1,3\}$    & $(3,U)$                  & $\{2,3,4\}$  & $(34,U)$  \\
$\{1,4\}$    & $(4,U)$                  & $U$          & $(U,U)$   \\ \hline
\end{tabular} 
\caption{\label{Table:Approximations2} Approximations based on $E$ and $T$ of Example~\ref{Example:NonDistributiveLattice}}
}
\end{table}

The approximations are given in Table~\ref{Table:Approximations2}
and the lattice $\textit{RS}(E,T)$ can be found in Figure~\ref{Fig:NonDistributive}.
This lattice  is not distributive, because 
\[(3,123) \vee ((3,U) \wedge (4,124)) = (3,123) \vee (\emptyset,\emptyset) = (3,123), \]
but 
\[ ( (3,123) \vee (3,U)) \wedge ( (3,123) \vee (4,124) ) = (3,U) \wedge (34,U) = (3,U). \]

\begin{figure}[ht]
\centering
\includegraphics[width=60mm]{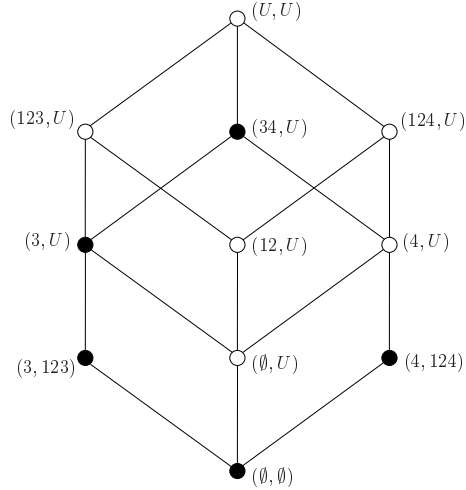}
\caption{The lattice $\mathit{RS}(E,T)$ of Example~\ref{Example:NonDistributiveLattice} is not
distributive, because it contains the pentagon $N_5$, marked by filled circles, as a sublattice.}
\label{Fig:NonDistributive}
\end{figure}
\end{example}

\section{Further properties of $\mathit{RS}(E,T)$} \label{Sec:FurtherProperties}

In case $T$ is a tolerance induced by an irredundant covering of $U$ and $E$ is an equivalence on $U$ such that $T$ is
$E$-compatible, we may present stronger lattice-theoretical results as in the previous section. As we already noted, 
it is proved in \cite{JarRad14} that if $T$ is a tolerance 
induced by an irredundant covering of $U$, then $\wp(U)_T$ and $\wp(U)^T$ are algebraic and completely distributive lattices. 
Since the Cartesian product of completely distributive and algebraic lattices is completely distributive and algebraic, $\wp(U)_E \times \wp(U)^T$ is 
completely distributive and algebraic whenever $T$ is a tolerance induced by an irredundant covering of $U$ and  
$E$ is an equivalence on $U$. Therefore, finding a condition under which $\mathit{RS}(E,T)$ is a complete sublattice of 
$\wp(U)_E \times \wp(U)^T$ would be important, because then we can show that $\mathit{RS}(E,T)$ has several further properties. 

Recall from Lemma~\ref{Lem:CompleteSublattice} that if $\Sigma_E \subseteq \Sigma_T$, $\mathit{RS}(E,T)$ is a complete 
sublattice of $\wp(U)_E \times \wp(U)^T$. Our following theorem characterizes when $\mathit{RS}(E,T)$ is a complete sublattice of 
$\wp(U)_E \times \wp(U)^T$ in terms of $\Sigma_E$ and $\Sigma_T$.

\begin{theorem} \label{Thm:Distributivity}
Let $T$ be a tolerance induced by an irredundant covering of $U$ and let $E$ be an equivalence on $U$ such that $T$ is $E$-compatible. 
Then the following are equivalent:
\begin{enumerate}[\rm (a)]
\item $\mathit{RS}(E,T)$ is a complete sublattice of $\wp(U)_E \times \wp(U)^T$.
\item For each $x \in \Sigma_E \setminus \Sigma_T$, there exists an element $y \notin \Sigma_E$ with $T(y) \subseteq T(x)$.
\end{enumerate}
\end{theorem}

\begin{proof}
(a)$\Rightarrow$(b):
Let $x \in \Sigma_E \setminus \Sigma_T$. Then $[x]_E = \{x\}$ and $T(x)$ has at least two elements. 
This means that there is $z \neq x$ such that $x \, T \, z$.
Because $T$ is induced by an irredundant covering $\mathcal{C}$, there is $B \in \mathcal{C}$ such that $\{x,z\} \subseteq B$.
For any $b \in B$, there is an $(E,T)$-rough set $(\{b\}_E,\{b\}^T)$. Let us assume that $\mathit{RS}(E,T)$ is a complete 
sublattice of $\wp(U)_E \times \wp(U)^T$. Then 
\[ \bigwedge_{b \in B} ( \{b\}_E,\{b\}^T) = \Big (\bigcap_{b \in B} \{b\}_E, \Big ( \Big ( \bigcap_{b \in B} \{b\}^T \Big )_T \Big )^T \Big ).\]
We have that 
\begin{center}
$\{x\}_E = \emptyset$ \quad or \quad $\{x\}_E = \{x\}$, 
\end{center}
and 
\begin{center}
$\{z\}_E = \emptyset$ \quad or \quad $\{z\}_E = \{z\}$.
\end{center}
Because $x \neq z$, we get $\bigcap \{ \{b\}_E \mid b \in B\} = \emptyset$.

By Section~\ref{SubSec:Coverings}, the know that $\mathcal{C} = \{ T(x) \mid \text{$T(x)$ is a block} \}$. Therefore, there exists
and element $c \in B$ such that $T(c) = \{c\}^T = B$. By \eqref{Eq:Block}, $B = \bigcap \{ T(b) \mid b \in B\} = \bigcap \{ \{b\}^T \mid b \in B\}$,
which yields
\[  \Big ( \Big ( \bigcap_{b \in B} \{b\}^T \Big )_T \Big )^T = (T(c)_T)^T = ((\{c\}^T)_T)^T = \{c\}^T =  T(c) = B .\]
We have that 
\[ \bigwedge_{b \in B} ( \{b\}_E,\{b\}^T) = (\emptyset,B) \]
belongs to  $\mathit{RS}(E,T)$. This means that there is a set $Y \subseteq U$ such that $Y_E = \emptyset$ and $Y^T = B$.
Obviously, $Y = \emptyset$ would imply $Y^T = \emptyset$, so necessarily $Y \neq \emptyset$. Thus, there is
$y \in Y$ such that $\{y\}_E \subseteq Y_E = \emptyset$ and $y \in T(y) = \{y\}^T \subseteq Y^T = B$. 
Because $y \in B$, we have also $B \subseteq T(y)$. This means that $T(y) = B$ is a block.
Now $\{y\}_E = \emptyset$ means that $[y]_E \neq \{y\}$, that is, $y \notin \Sigma_E$. Because $x \in B = T(y)$ and 
$T(y)$ is a block, we have $T(y) \subseteq T(x)$. 

\noindent%
(b)$\Rightarrow$(a):
By Theorem~\ref{Thm:CompleteLattice}  $\mathit{RS}(E,T)$ is a complete join-sublattice  of
$\wp(U)_E \times \wp(U)^T$. By the same theorem, to prove that $\mathit{RS}(E,T)$ is a complete sublattice  of
$\wp(U)_E \times \wp(U)^T$, we have to show that for any $\mathcal{H} \subseteq \wp(U)$,
\begin{equation}\label{Eq:HaveTo}
 \Big ( \Big (\bigcap_{X \in \mathcal{H}} (X^T)_T \Big) \setminus \Sigma_E(\mathcal{H}) \Big) ^T  =
 \Big (\bigcap_{X \in \mathcal{H}} (X^T)_T \Big) ^T  .
\end{equation}
Let $\mathcal{H} \subseteq \wp(U)$. Because the left side of \eqref{Eq:HaveTo} is always included in its right side of, it is enough to prove that 
\[
 \Big (\bigcap_{X \in \mathcal{H}} (X^T)_T \Big) ^T \subseteq
 \Big ( \Big (\bigcap_{X \in \mathcal{H}} (X^T)_T \Big) \setminus \Sigma_E(\mathcal{H}) \Big) ^T .
\] 
Let $x \in \big ( \bigcap \{ (X^T)_T \mid X \in \mathcal{H} \} \big )^T$. This means that there is 
$a \in  \bigcap \{ (X^T)_T \mid X \in \mathcal{H} \} = \big ( \bigcap \{ X^T \mid X \in \mathcal{H} \} \big )_T$ with $x \in T(a)$. 
As $x \, T \, a$, there is $b \in U$ such that $T(b)$ is a block and $x,a \in T(b)$. We have
$x \in T(b) \subseteq T(a) \subseteq \bigcap \{ X^T \mid X \in \mathcal{H} \}$.

If $b \notin \Sigma_E(\mathcal{H})$, then trivially 
$b \notin \big ( \bigcap \{ (X^T)_T \mid X \in \mathcal{H} \} \setminus \Sigma_E(\mathcal{H}) \big )$
and $x \in \big (\bigcap \{ (X^T)_T \mid X \in \mathcal{H} \} \setminus \Sigma_E(\mathcal{H}) \big )^T$. 
Assume that $b \in \Sigma_E(H)$. Let us recall that
\[ \Sigma_E(\mathcal{H}) = \Big (  \Big(  \bigcap_{X\in\mathcal{H}} ( X^T)_T \Big )  
\setminus \bigcap_{X\in\mathcal{H}} X_E   \Big ) \cap \Sigma_E.
\]
Then $[b]_E = \{b\}$ and $b \in \bigcap \{(X^T)_T \mid X \in \mathcal{H} \} = 
(\bigcap \{X^T \mid X \in \mathcal{H} \})_T$. This yields 
$T(b) \subseteq \bigcap \{ X^T \mid X \in \mathcal{H} \}$. In addition, we have $b \notin \bigcap \{ X_E \mid X \in \mathcal{H} \}$.
Observe also that for all $X \in \mathcal{H}$, $b \in X^T$, that is, $T(b) \cap X \neq \emptyset$.
This implies that $T(b) = \{b\}$ is not possible, because it would imply that $b \in X$ for all $X \in \mathcal{H}$. Since $[b]_E = \{b\}$,
we would have that $b \in X_E$ for all $X \in \mathcal{H}$, and further $x \in \bigcap \{ X_E \mid X \in \mathcal{H} \}$, which is not allowed.
Hence, $T(b) \neq \{b\}$ and $b \in \Sigma_E \setminus \Sigma_T$. By our assumption, there exists an element $y$ such that $T(y) \subseteq T(b)$
and $[y]_E \neq \{y\}$. Now $T(y) \subseteq T(b) \subseteq \bigcap \{ X^T \mid X \in \mathcal{H} \}$ implies
$y \in (\bigcap \{ X^T \mid X \in \mathcal{H} \})_T = \bigcap \{ (X^T)_T \mid X \in \mathcal{H} \}$. 
Because $y \notin \Sigma_E$, $y \notin \Sigma_E(\mathcal{H})$ holds also. We have showed that 
$y \in \bigcap \{ (X^T)_T \mid X \in \mathcal{H} \} \setminus \Sigma_E(\mathcal{H})$.
Because $y \in T(b)$ and $T(b)$ is a block, we have $x \in T(b) \subseteq T(y)$ and thus
\[ x \in \Big ( \Big (\bigcap_{X \in \mathcal{H}} (X^T)_T \Big) \setminus \Sigma_E(\mathcal{H}) \Big) ^T ,\]
which completes the proof.
\end{proof}

\begin{remark}
In condition (b) of Theorem~\ref{Thm:Distributivity}, any $y \notin \Sigma_E$ with $T(y) \subseteq T(x)$ is such that $T(y)$ is a
block. In ``(a)$\Rightarrow$(b) part'' of the proof, it is directly showed that $T(y)$ is a block.
In ``(b)$\Rightarrow$(a) part'', we showed that $T(y) \subseteq T(b)$, and $T(b) \subseteq T(y)$ holds by assumption.
Thus, $T(y) = T(b)$, and we also have that $T(b)$ is a block.

What also is interesting is that there need to be \emph{two} such elements. Namely, $y \notin \Sigma_E$
means that $[y]_E \neq \{y\}$. So, there is an element $z \neq y$ such that $y \, E \, z$. Because $E$ is $T$-compatible, we have that
$(y,z) \in \ker T$, that is, $T(y) = T(z)$. Hence, $z$ has all the same properties as $y$ has.
\end{remark}

\begin{example}
Suppose that $U = \{1,2,3,4,5,6\}$. Then
\[ \mathcal{C} = \{ \{1,2,3,4\}, \{3,4,5,6\}\}\] 
is an irredundant covering of $U$,
\begin{itemize}
 \item $T(1) = T(2) =  \{1,2,3,4\}$,
 \item $T(3) = T(4) = U$, and
 \item $T(5) = T(6) = \{3,4,5,6\}$.
\end{itemize}
We have that 
\[ U / \ker T =  \{ \{1,2\}, \{3,4\}, \{5,6\} \}.\]
Let $E$ be an equivalence on $U$ such that
\[ U/E =  \{ \{1,2\}, \{3\}, \{4\}, \{5,6\} \} .\]
Because $E \subseteq \ker T$, the tolerance $T$ is $E$-compatible.

Now  $\Sigma_T = \emptyset$ and $\Sigma_E = \{3,4\}$. Thus, $\Sigma_E \setminus \Sigma_T = \{3,4\}$. Because 
$T(3) = T(4) = U$, condition (b) of Theorem~\ref{Thm:Distributivity} is trivially true for any 
$y \in U \setminus \Sigma_E = \{1,2,5,6\}$. Note also that for any element $y \in  \{1,2,5,6\}$,
$T(y)$ is a block.
\end{example}

However, we may present even such a condition concerning only the elements whose $T$-neighbourhood is a block. 
Let us define the following condition:
\begin{tabbing}\qquad\qquad\= \\
(CSub) \>For each $x \in \Sigma_E \setminus \Sigma_T$ such that $T(x)$ is a block, there exists an element  $y \notin \Sigma_E$ with $T(y) = T(x)$.
\end{tabbing}
Note that  if $\Sigma_E \subseteq \Sigma_T$, then $\Sigma_E \setminus \Sigma_T$ is empty and (CSub) holds trivially.

\begin{lemma} \label{Lem:SimpleCondition}
Let $T$ be a tolerance induced by an irredundant covering of $U$ and let $E$ be an equivalence on $U$ such that $T$ is $E$-compatible. 
Then $\mathit{RS}(E,T)$ is a complete sublattice of $\wp(U)_E \times \wp(U)^T$ if and only if condition\/ {\rm (CSub)} holds.
\end{lemma}

\begin{proof}
We prove that (CSub) is equivalent to condition (b) of Theorem~\ref{Thm:Distributivity}, from which the claim follows. 

Assume that $x \in \Sigma_E \setminus \Sigma_T$. Then, there exists an element $y \notin \Sigma_E$ with $T(y) \subseteq T(x)$.
If this $T(x)$ is a block, then $y \in T(y) \subseteq T(x)$ implies $T(x) \subseteq T(y)$. Thus, $T(y) = T(x)$.

Conversely, assume that (Csub) holds and $x \in \Sigma_E \setminus \Sigma_T$. Then $T(x) \neq \{x\}$, that is, $x \, T \, z$ for some
$z \neq x$. Hence, there exists $y$ such that $T(y)$ is a block and $x,z \in T(y)$. Because $x \in T(y)$, we have $T(y) \subseteq T(x)$.
If $y \notin \Sigma_E$, then there is nothing left to prove. If $y \in \Sigma_E$, then $x,z \in T(y)$ and $x \neq z$ imply that $T(y)$
has at least two elements. Thus, $y \in \Sigma_E \setminus \Sigma_T$, and by (CSub), there exists an element
$y' \notin \Sigma_E$ such that $T(y') = T(y) \subseteq T(x)$. This completes the proof.  
\end{proof}

It is proved in \cite{JarRad18} that when $T$ is a tolerance induced by an irredundant covering of $U$, 
$\mathit{RS}(T)$ is a regular double $p$-algebra. Recall from \cite{Kat73}, for example, that
an algebra $(L,\vee,\wedge,{^*},{^+},0,1)$ is called a \emph{double $p$-algebra} if $(L,\vee,\wedge,0,1)$ is
a bounded lattice such that $^*$ is the \emph{pseudocomplement} operation and $^+$ is the \emph{dual pseudocomplement operation} on $L$. 
Note that this means that for all $a \in L$, $a \wedge b = 0$ if and only if $b \leq a^*$ and
$a \vee b = 1$ if and only if $b \geq a^+$. A double $p$-algebra is regular if
\[ a^* = b^* \quad \text{and} \quad a^+ = b^+ \quad \text{imply} \quad a=b .\]

A \emph{Boolean lattice} is a bounded distributive lattice $L$ such that each element $a \in L$ has
a \emph{complement} $a'$ which satisfies
\[ a \wedge a' = 0 \quad \text{and} \quad a \vee a' = 1. \]
Note that a Boolean lattice $B$ forms trivially a regular double $p$-algebra $(B,\vee,\wedge,{'},{'},0,1)$.

In the proof of the following proposition we need the fact that $\mathrm{Def}(E)$ is a Boolean lattice in 
which $X' = X^c$ for all $X \in \mathrm{Def}(E)$. In addition, it is proved in \cite{JarRad14} that if $T$ is
a tolerance induced by an irredundant covering of $U$, then $\wp(U)^T$ is a Boolean lattice
such that $X' = (X^c)^T$ for $X \in \wp(U)^T$.

\begin{proposition} \label{Prop:DoublePseudocomplement}
Let $T$ be a tolerance induced by an irredundant covering of $U$ and let $E$ be an equivalence on $U$ such that
$T$ is $E$-compatible. If {\rm (CSub)} holds, then
\[ (\mathit{RS}(E,T),\vee,\wedge,^*,^+,(\emptyset,\emptyset), (U,U) )\]
is a double $p$-algebra such that for any $(A,B),(C,D) \in \mathit{RS}(E,T)$,
\begin{align*}
(A,B) \wedge (C,D) &= (A \cap C, ((B \cap D)_T)^T ), \\ 
(A,B) \vee   (C,D) &= (A \cup C, B \cup D) ), \\
(A,B)^* &= ( ((B^c)^T)_T, (B^c)^T  ), \\
(A,B)^+ &= ( A^c, (A^c)^T  ). 
\end{align*}
\end{proposition}

\begin{proof} The operations $\vee$ and $\wedge$ are clear, because they are inherited from $\wp(U)_E \times \wp(U)^T$.

Let $(A,B) \in \mathit{RS}(E,T)$. Then $A \in \mathrm{Def}(E)$ gives and $A^c \in  \mathrm{Def}(E)$ and 
$(A^c)_E = A^c$. We have that 
\[ ( A^c, (A^c)^T  ) = ( (A^c)_E, (A^c)^T  ) \in  \mathit{RS}(E,T).\]
By Lemma~\ref{Lem:AllTheSame}(b), $(((B^c)^T)_T)_E = ((B^c)^T)_T$. Since $T$ is a tolerance, $(((B^c)^T)_T)^T = (B^c)^T$. Thus also
\[ ((B^c)^T)_T, (B^c)^T) = (((B^c)^T)_T)_E, ((B^c)^T)_T)^T) \in  \mathit{RS}(E,T).\]

Now 
\begin{equation} \label{Eq:PseudoMeet}
(A,B) \wedge ( ((B^c)^T)_T, (B^c)^T  ) = ( A \cap ((B^c)^T)_T, (B_T \cup ((B^c)^T)_T)^T  ) . 
\end{equation}
Because $A = X_E$ and $B = X^T$ for some $X \subseteq U$, $A^T = (X_E)^T \subseteq X^T = B$ and 
$B^c \subseteq (A^T)^c = (A^c)_T$. Then,
\[ ((B^c)^T)_T \subseteq (((A^c)_T)^T)_T = (A^c)_T \subseteq A^c , \]
and we obtain $A \cap ((B^c)^T)_T = \emptyset$. In addition, 
\[ B_T \cap ((B^c)^T)_T = B_T \cap ((B_T)^c)_T \subseteq B_T \cap (B_T)^c = \emptyset, \]
from which we get $(B_T \cap ((B^c)^T)_T)^T = \emptyset$. Thus, by \eqref{Eq:PseudoMeet},
\[ (A,B) \wedge ( ((B^c)^T)_T, (B^c)^T  )  = ( A \cap ((B^c)^T)_T, (B_T \cup ((B^c)^T)_T)^T  ) = (\emptyset, \emptyset). \]

Suppose that $(A,B) \wedge ( Y_E, Y^T  ) = (\emptyset, \emptyset)$. Because  $\mathit{RS}(E,T)$ is a 
complete sublattice of $\wp(U)_E \times \wp(U)^T$, 
\begin{center}
$A \wedge Y_E = \emptyset$ in $\mathrm{Def}(E)$ and $B \wedge Y^T = \emptyset$ in $\wp(U)^T$. 
\end{center}
This gives that $Y_E \subseteq A^c$ and $Y^T \subseteq (B^c)^T$. Recall from the above that $\wp(U)^T$
is a Boolean lattice in which $B' = (X^c)^T$. Then, $Y_E \subseteq Y \subseteq (Y^T)_T \subseteq ((B^c)^T)_T$ 
and hence,
\[ (Y_E, Y^T) \leq ( ((B^c)^T)_T, (B^c)^T).\]
We have proved $(A,B)^* = ( ((B^c)^T)_T, (B^c)^T  )$. 

For the other equality, we have that
\[ (A,B) \vee (A^c, (A^c)^T) = (A \cup A^c, B \cup (A^c)^T).\]
Now $A \cup A^c = U$ and $A \subseteq B$ gives $B^c \subseteq A^c$. 
Thus, $B \cup (A^c)^T \supseteq B \cup A^c \supseteq B \cup B^c = U$. Therefore,
$(A,B) \vee (A^c, (A^c)^T) = (U,U)$. On the other hand, assume that
$(A,B) \vee (Y_E,Y^Y) = (U,U)$. Because $\mathit{RS}(E,T)$ is a complete sublattice of $\wp(U)_E \times \wp(U)^T$, 
we have that $A \vee Y_E = U$ in $\mathrm{Def}(E)$. Therefore,
$A^c \subseteq Y_E$ and $A^c \subseteq Y_E \subseteq Y$ gives $(A^c)^T \subseteq Y^T$.
This means that $(A_c,(A^c)^T) \leq (Y_E,Y^T)$. We get $(A,B)^+ = ( A^c, (A^c)^T  )$.
\end{proof}

\begin{corollary} \label{Corollary:PseudoPseudo}
Let $T$ be a tolerance induced by an irredundant covering of $U$ and let $E$ be an equivalence on $U$ such that
$T$ is $E$-compatible. If {\rm (CSub)} holds, then for $(A,B) \in \mathit{RS}(E,T)$,
\[ (A,B)^{**} = (B_T,B) \quad \text{and} \quad (A,B)^{++} = (A, A^T).\]
\end{corollary}

\begin{proof} By Proposition~\ref{Prop:DoublePseudocomplement},
\[
(A,B)^{**} = ( ((B^c)^T)_T, (B^c)^T)^* = (( (((B^c)^T)^c)^T)_T, (((B^c)^T)^c)^T) 
         = ( ((B_T)^T)_T,(B_T)^T) = (B_T, (B_T)^T).
\]
Because $(A,B) = (X_E,X^T)$ for some $X \subseteq U$, we have $(B_T)^T = ((X^T)_T)^T = X^T = B$ and $(A,B)^{**} = (B_T,B)$.
Similarly,
\[ (A,B)^{++} = (A^c, (A^c)^T)^* =  ( (A^c)^c, ((A^c)^c)^T ) = (A,A^T). \qedhere \] 
\end{proof}

\begin{theorem} \label{Theorem:Regularity}
Let $T$ be a tolerance induced by an irredundant covering of $U$ and let $E$ be an equivalence on $U$ such that
$T$ is $E$-compatible. If {\rm (CSub)} holds, then
\[ (\mathit{RS}(E,T),\vee,\wedge,^*,^+,(\emptyset,\emptyset), (U,U) )\]
is a regular double $p$-algebra.
\end{theorem}

\begin{proof}
Suppose $(A,B)^* = (C,D)^*$ and $(A,B)^+ = (C,D)^+$ for some $(A,B)$ and $(C,D)$ in  $\mathit{RS}(E,T)$.
Then,
\[ (A,B)^{**} = (B_T,B) = (D_T,D) = (C,D)^{**} \] 
and
\[ (A,B)^{++} = (A,A^T) = (C,C^T) = (C,D)^{++}.\]
We have $B = D$ and $A = C$, that is, $(A,B) = (C,D)$.
\end{proof}

A \emph{Heyting algebra} is a bounded lattice $L$ such that for all $a,b\in L$, there is a greatest element $x$ of $L$ satisfying
$a \wedge x \leq b$. This element $x$ is called the \emph{relative pseudocomplement}  of $a$ with respect to $b$, 
and it is denoted by $a \Rightarrow b$.  
By \cite[Theorem~1]{Kat73}, we can write the following corollary of Theorem~\ref{Theorem:Regularity}.

\begin{corollary} \label{Corollary:Heyting}
Let $T$ be a tolerance induced by an irredundant covering of $U$ and let $E$ be an equivalence on $U$ such that
$T$ is $E$-compatible. If {\rm (CSub)} holds, then $\mathit{RS}(E,T)$ is a Heyting algebra
\end{corollary}
\noindent%
Note that if an $E$-compatible tolerance $T$ is induced by an irredundant covering and (CSub) holds, then
$\mathit{RS}(E,T)$ is a Heyting algebra also because it is completely distributive.

It is also proved in \cite{JarRad18} that if $T$ is a tolerance induced by an irredundant covering of $U$,
then $\mathit{RS}(T)$ forms so-called De~Morgan algebra. A \emph{De~Morgan algebra} $(L,\vee,\wedge,{\sim},0,1)$
is a bounded distributive lattice $(L,\vee,\wedge,0,1)$ equipped with an operation $\sim$ which
satisfies:
\begin{center}
${\sim}{\sim} x =x$  \quad and \quad  $x\leq y \iff {\sim} y \leq {\sim} x$.
\end{center}
Such a map $\sim$ is an order-isomorphism from $(L,\leq)$ to $(L,\geq)$.
This means that the Hasse diagram of $L$ looks the same when it is  turned upside-down.

\begin{example} \label{Exa:NonDual}
Let $U = \{1,2,3,4,5,6\}$ and let $T$ be the tolerance induced by the irredundant covering
\[ \mathcal{H} = \{ \{1,2,3,4\} , \{3,4,5,6\} \} \]
of $U$. We have 
\[ T(1) = T(2) =  \{1,2,3,4\}, \ T(3) = T(4) = U, \ T(5) = T(6) =  \{3,4,5,6\}.\]
Let $E = \ker T$. Then, $U/E = \{ \{1,2\}, \{3,4\}, \{5,6\}\}$. This means that 
$\Sigma_E = \emptyset$ and $\mathit{RS}(E,T)$ is a complete sublattice of $\wp(U)_E \times \wp(U)^T$.
By Theorem~\ref{Theorem:Regularity},
\[ (\mathit{RS}(E,T),\vee,\wedge,^*,^+,(\emptyset,\emptyset), (U,U) )\]
is a regular and distributive double $p$-algebra.

The Hasse diagram of $\mathit{RS}(E,T)$ is given in Figure~\ref{Fig:NonMorgan}. Because
$\mathit{RS}(E,T)$ is not isomorphic to its dual, whose Hasse diagram is obtained by
turning the Hasse diagram of $\mathit{RS}(E,T)$ upside down, $\mathit{RS}(E,T)$ cannot form a De~Morgan algebra.

\begin{figure}[ht]
\centering
\includegraphics[width=60mm]{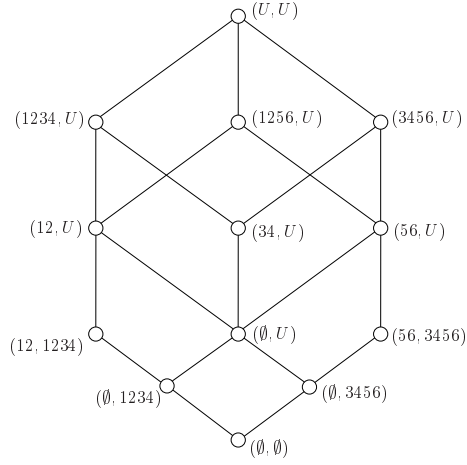}
\caption{A regular and distributive double $p$-algebra $\mathit{RS}(E,T)$ of Example~\ref{Exa:NonDual}
does not form a De~Morgan algebra, because it is  not isomorphic to its dual.}
\label{Fig:NonMorgan}
\end{figure}
\end{example}

We end this work by considering the case in which $E$ is an equivalence on $U$ and the $E$-compatible tolerance is an equivalence.
Suppose that $F$ is an equivalence on $U$ such that $E \subseteq \ker F$. Because $F$ is an equivalence, $\ker F = F$ and we have 
$E \subseteq F$. On the other hand, if $E \subseteq F$, then $E \subseteq \ker F$. Thus, whenever $F$ is an equivalence,  
$F$ is $E$-compatible if and only if $E \subseteq F$. Notice that this means that $F$-classes are unions of $E$-classes.

\pagebreak%

Let $E$ and $F$ be two equivalences on $U$ such that $E \subseteq F$. The ``tolerance'' $F$ obviously is induced by an irredundant covering 
$U/F$ and it is compatible with $E$. Let us introduce the following condition:
\begin{tabbing}\qquad\qquad\= \\
(CSub$^\circ$) \>If $[x]_F$ is non-singleton, then there is $y \, F \, x$ such that $[y]_E$ is non-singleton.
\end{tabbing}

\begin{lemma}\label{lem:nonsingle}
Let $E$ and $F$ be two equivalences on $U$ such that $E \subseteq F$. Then $\mathit{RS}(E,F)$ is a 
complete sublattice of $\wp(U)_E \times \wp(U)^F$ if and only condition\/ {\rm (CSub$^\circ$)} holds.
\end{lemma}

\begin{proof} We show that (CSub$^\circ$) is equivalent to (CSub), when we replace $T$ by $F$ in (CSub).

\smallskip\noindent%
(CSub)$\Rightarrow$(CSub$^\circ$): 
Suppose that $[x]_F$ is non-singleton. If $[x]_E$ is non-singleton, we may choose $y = x$.
If $[x]_E$ is singleton, then $x \in \Sigma_E \setminus \Sigma_F$. The equivalence class $[x]_F$ is a block.
By (CSub), there exists $y \notin \Sigma_E$ such that $[y]_F = [x]_F$. This means that $y \, F \, x$ and $[y]_E$ is non-singleton.

\smallskip\noindent%
(CSub$^\circ$)$\Rightarrow$(CSub): Let $x \in \Sigma_E \setminus \Sigma_F$. Then $[x]_F \neq \{x\}$ is an equivalence class and a block.
There exists  $y \, F \, x$ such that $[y]_E$ is non-singleton, that is, $y \notin \Sigma_E$. 
We have $[x]_F = [y]_F$, because $F$ is an equivalence.
\end{proof}

If (CSub$^\circ$) holds, then $\mathit{RS}(E,F)$ is a complete sublattice of the lattice $\wp(U)_E \times \wp(U)^F$ forming a 
distributive double $p$-algebra such that for $(A,B) \in \mathit{RS}(E,F)$,
\begin{align*}
(A,B)^*    &= ( (B^c)^F, (B^c)^F  ), &        (A,B)^+    &= ( A^c, (A^c)^F  ), \\ 
(A,B)^{**} &= (B_F,B),                 &      (A,B)^{++} &= (A, A^F). 
\end{align*}

A \emph{Stone algebra} is a pseudo-complemented distributive lattice $(L,\vee,\wedge,^*,0,1)$ such that $a^* \vee a^{**} = 1$ for all $a \in L$.
A \emph{double Stone algebra} is a double $p$-algebra $(L,\vee,\wedge,{^*},{^+},0,1)$ such that 
$a^* \vee a^{**} = 1$ and $a^+ \wedge a^{++} = 0$ for all $a \in L$.

\begin{proposition}\label{Prop:Last}
Let $E$ and $F$ be equivalences on $U$ such that $E \subseteq F$ and {\rm (CSub$^\circ$)} holds. 
\begin{enumerate}[\rm (a)]
 \item $(\mathit{RS}(E,F),\vee,\wedge,^*,(\emptyset,\emptyset), (U,U) )$ is a Stone algebra. 
 \item $(\mathit{RS}(E,F),\vee,\wedge,^*,^+,(\emptyset,\emptyset), (U,U) )$ is a double Stone algebra if and only if $E = F$.
\end{enumerate}
\end{proposition}

\begin{proof}
(a) For $(A,B) \in \mathit{RS}(E,F)$,
\[
(A,B)^* \vee (A,B)^{**} = ( (B^c)^F, (B^c)^F  ) \vee (B_F, B) = ( (B_F)_c \cup B_F, (B_F)^c \cup B) = (U,U).
\]
This is because $B_F \subseteq B$ gives $B^c \subseteq (B_F)^c$ and so $(B_F)^c \cup B \supseteq B^c \cup B = U$.

(b) If $E = F$, then $\mathit{RS}(E,F)$ coincides to the rough set algebra $\mathit{RS}(E)$.
It is well known that $\mathit{RS}(E)$ forms a double Stone algebra; see \cite{Com93, PomPom88}. 
Conversely, assume that $\mathit{RS}(E,F)$ forms a double Stone algebra. Then
\[
(A,B)^+ \wedge (A,B)^{++} = ( A^c, (A^c)^F  ) \wedge (A, A^F) = (A^c \cap A, (A^c)^F \cap A^F) = (\emptyset,\emptyset)
\]
holds for all $(A,B) \in \mathit{RS}(E,F)$. Because $A^c \cap A = \emptyset$ holds trivially, we have that
this is equivalent to
\begin{equation} \label{Eq:CondStone}
 (A^c)^F \cap A^F = \emptyset \text{ \ for all $A \in \mathrm{Def}(E)$}.
\end{equation}

We prove that \eqref{Eq:CondStone} is equivalent to $\mathrm{Def}(E) = \mathrm{Def}(F)$. 
Suppose that $A \in \mathrm{Def}(E) = \mathrm{Def}(F)$.
Then $A^c \in \mathrm{Def}(E) = \mathrm{Def}(F)$ and we get
\[ (A^c)^F \cap A^F = A^c \cap A = \emptyset, \]
that is, \eqref{Eq:CondStone} holds. 

Conversely, suppose that \eqref{Eq:CondStone} holds. Since $E \subseteq F$, $F$-classes are unions of $E$-classes.
Thus, $\mathrm{Def}(F) \subseteq \mathrm{Def}(E)$. Let $A \in \mathrm{Def}(E)$. Then \eqref{Eq:CondStone} implies
$A^F \subseteq ((A^c)^F)^c = A_F$. Because $A_F \subseteq A^F$, we get $A_F = A^F$, that is, $A \in \mathrm{Def}(F)$.
Thus, $\mathrm{Def}(E) = \mathrm{Def}(F)$. 

Finally, in order to prove $E = F$, we show $[x]_E = [x]_F$ for any $x \in U$. Suppose that $x \in U$. Because $E \subseteq F$,
we have $[x]_E \subseteq [x]_F$. Let $a \in [x]_F$. We have $x \in [x]_E$ and $x \in [a]_F$. Thus $[x]_E \cap [a]_F \neq \emptyset$
and so $a \in ([x]_E)^F$. Because $[x]_E \in \mathrm{Def}(E) = \mathrm{Def}(F)$, we have $([x]_E)^F = [x]_E$ and $a \in [x]_E$.
Hence also $[x]_F \subseteq [x]_E$ and we have proved $E = F$.
\end{proof}

\section*{Some concluding remarks}

In this paper we have presented observations on a tolerance $T$ compatible with an equivalence $E$.
Surprisingly, this notion was defined in the literature under a different name already in 1995.
Our opinion is that since this concept appears in several contexts, it ``proves'' that the notion is important. 
Our motivation for defining this concept was that we wanted to make the upper approximation $X^E$ of a set $X$ ``softer''. 
A tolerance $T$ compatible with $E$ turned out to be a suitable for this, because it connects the tolerance 
$T$ to the equivalence $E$ firmly, so that the connections between approximations defined in terms of $T$ and $E$ 
are not arbitrary. In \cite{Slowinski95}, the authors ``extend'' the equivalence $E$ by accepting that objects which are not
indistinguishable but sufficiently close or similar can be ``grouped'' together. More precisely, their
aim was to construct a similarity relation (tolerance) from an indistinguishability relation (equivalence). 

In our study it turned out that tolerances compatible with an equivalence $E$ are closely related to $E$-definability. 
We proved that if $T$ is a tolerance compatible with an equivalence $E$ on $U$, then lower and upper $T$-approximations are 
$E$-definable. This also implies that for each $x \in U$, the neighbourhood $T(x)$ is a union of some $E$-classes.
A block of a tolerance can be seen as a counterpart of an equivalence class of an equivalence relation.
Blocks  are maximal sets in which all elements are similar to each other. We have proved that all $T$-blocks are 
$E$-definable if and only if $T$ is $E$-compatible.

The ordered set of all rough sets $\mathit{RS}(E)$ defined by an equivalence $E$ is known to form a regular double 
Stone algebra. On the other hand, rough sets $\mathit{RS}(T)$ defined by a tolerance $T$ do not necessarily form even a 
semilattice. In this work, we have shown that $\mathit{RS}(E,T)$ forms a complete lattice,
whenever $T$ is $E$-compatible. In addition, if an $E$-compatible tolerance $T$ is induced by an irredundant covering,
we have given a condition under which $\mathit{RS}(E,T)$ forms a regular double $p$-algebra and a Heyting algebra.

\section*{Acknowledgements}
We thank the anonymous referees for the careful reading of our manuscript and their comments and suggestions.
The work of the second and third author was carried out as a part of the EFOP-3.6.1-16-00011 ``Younger and
Renewing University -- Innovative Knowledge City'' project implemented in the framework of the
Sz{\'e}chenyi 2020 program, supported by the European Union, co-financed by the European Social Fund.



\end{document}